\documentclass[preprint, number]{elsarticle}

\RequirePackage{amsthm,amsmath,natbib}
\RequirePackage[colorlinks,citecolor=blue,urlcolor=blue]{hyperref}
\RequirePackage{hypernat}

\newcommand\dela[1]{\Green{{\small{30.01.2010}}}}

\newcommand\delb[1]{\Green{{\small{13.02.2010}}}}

\newcommand\delc[1]{\Green{{\small{21.02.2010}}}}

\newcommand\deld[1]{\Red{{\small{15.03.2010}}}}

\theoremstyle{plain}

\usepackage{amsfonts,amssymb,amsmath,graphicx}
\input colordvi
\newcommand\del[1]{}
\newcommand\Greendel[1]{}
\newcommand\old[1]{}
\newcommand\new[1]{#1}

\numberwithin{equation}{section}
\newcommand{\embed}{\hookrightarrow}
\newcommand\Cer{C_{1/2}^\xi(\mathbb{R})}
\newcommand\Oer{\Omega(\xi)}

\newcommand\Ce{C_{1/2}^\xi(\mathbb{R},\mathrm{E})}
\newcommand\Cx{C_{1/2}^\xi(\mathbb{R},\mathrm{X})}
\newcommand\Oe{\Omega(\xi,\mathrm{E})}
\newcommand\Ox{\Omega(\xi,\mathrm{X})}

\newcommand\vnl{(\!(}
\newcommand\vnr{)\!)}
\newcommand\Vertt{|\!|}

\newcommand{\lb}{\langle}
\newcommand{\rb}{\rangle}

\newtheorem{theorem}{Theorem}[section]
\newtheorem{corollary}[theorem]{Corollary}
\newtheorem{lemma}[theorem]{Lemma}

\newtheorem{remark}[theorem]{Remark}

\newtheorem{definition}[theorem]{Definition}
\newtheorem{proposition}[theorem]{Proposition}

\newtheorem{Notation}[theorem]{Notation}

\newtheorem{Assumption}{Assumption}[section] 

\begin{document}
\begin{frontmatter}
\title{Random attractors for
stochastic $2D$-Navier-Stokes equations in some unbounded domains}

\author[zb]{Z. Brze{\'z}niak\corref{cor1}\fnref{u1}}\ead[label=e1]{zb500@york.ac.uk}

\author[sev]{T. Caraballo\fnref{t1}} \ead[label=e2]{caraball@us.es}

\author[sev]{J.A. Langa\fnref{t1}}
  \ead[label=e3]{langa@us.es}

\author[yl]{Y. Li\fnref{t2}}\ead[label=e4]{liyuhong@hust.edu.cn}

\author[gl]{G.~\L ukaszewicz\fnref{dag1}}
\ead[label=e5]{glukasz@mimuw.edu.pl}

\author[sev]{and J.~Real\fnref{t1}}\ead[label=e6]{jreal@us.es}

\cortext[cor1]{Corresponding author}

\fntext[u1]{Partially supported  by an EPSRC grant number
EP/E01822X/1}

\fntext[t1]{Partially supported by Ministerio de Ciencia e
Innovaci\'on (Spain) under grant MTM2008-00088, and
Consejer\'{\i}a de Innovaci\'{o}n, Ciencia y Empresa (Junta de
Andaluc\'{\i}a, Spain) under grants 2007/FQM314 and HF2008-0039}

\fntext[t2]{Partially Supported by Major Research Plan Program of National Natural Science Foundation of China (91130003), the SRF for the
ROCS, SEM of China, the Talent Recruitment Foundation of HUST}

\fntext[dag1]{Supported by Polish Government grant MEiN N201
547638 and EC Project FP6 EU SPADE2}

\address[zb]{Department of Mathematics, The University of
York, Heslington, York, Y010 5DD, UK }

\address[sev]{Dpto. de Ecuaciones Diferenciales  y An\'alisis
Num\'erico, Universidad de Sevilla, Apdo. de Correos 1160,  41080
Sevilla, Spain}

\address[yl]{Information Engineering  and Simulation
Centre,  Huazhong University of Science and Technology, Wuhan
430074, China}

\address[gl]{Institute of Applied Mathematics and Mechanics, University of
Warsaw, Banacha~2, 02-097 Warsaw, Poland}

\begin{abstract} We show that the stochastic flow generated by the $2$-dimensional Stochastic Navier-Stokes equations \textcolor{red}{with rough noise on} a Poincar{\'e}-like domain has a unique random attractor. \textcolor{red}{One of the technical problems associated with the rough noise is overcomed by the use of the corresponding Cameron-Martin (or reproducing kernel Hilbert)  space}. Our results complement the result by Brze{\'z}niak and Li \cite{Brz_Li_2006b} who showed that the corresponding flow is asymptotically compact and also generalize Caraballo et al. \cite{Caraballo_L_R_2006} who proved existence of
a unique \del{pullback} attractor for the time-dependent deterministic Navier-Stokes equations.
\end{abstract}


\end{frontmatter}

\section{Introduction}\label{sec-1}

The analysis of infinite dimensional Random Dynamical Systems
(RDS) is now an important branch in the study of qualitative
properties of stochastic PDEs. From the first papers of
Brze{\'z}niak et al. \cite{Brz_Cap_Fland_1993}, and  Crauel and
Flandoli \cite{Crauel_F_1994}, the use of the notions of random
and \del{pullback} attractors have been used in many papers to give
crucial information on the asymptotic behaviour of random
(Brze{\'z}niak et al. \cite{Brz_Cap_Fland_1993}), stochastic
(Arnold \cite{Arnold_1998}, Crauel and Flandoli
\cite{Crauel_F_1994}, Crauel \cite{Crauel_1995}) and
non-autonomous PDEs (Schmalfuss \cite{Schmalfuss}, Kloeden and
Schmalfuss \cite{Kloeden}, Caraballo et al.
\cite{Caraballo_L_R_2006}). Given a probability space, a random
attractor is a compact random set, invariant for the associated
RDS and attracting every bounded random set in its basis of
attraction (see Definition \ref{randomattractor}).

The main general result on  random attractors relies heavily on the
existence of a random compact attracting set, see Crauel et al.
\cite{Crauel4}. But this condition was only shown to be true  when the embedding $V\embed H$ is compact, i.e. when our
stochastic PDE is set in a \textit{bounded} domain. In the
deterministic case, this difficulty   was solved by different methods,
see Abergel \cite{Ab}, Ghidaglia \cite{Ghid-94} or Rosa
\cite{Rosa-98} for the autonomous case and \L ukaszewicz and
Sadowski \cite{lusa04} or Caraballo et al.
\cite{Caraballo_L_R_2006} for the non-autonomous one. Recently,
these methods have been also generalized to a stochastic framework, see Brze{\'z}niak and Li \cite{Brz_Li_2004}, \cite{Brz_2006},\cite{Brz_Li_2006b}, Bates
et al. \cite{bates2006},\cite{bates2009}, Wang \cite{wang2009}. In
particular, in Brze{\'z}niak and Li \cite{Brz_Li_2006b}, a deep
work is provided for the existence of a stochastic flow and its
asymptotic behaviour related to a 2D stochastic Navier-Stokes
equations in an unbounded domain with a very general irregular
additive white noise. Moreover, in \cite{Brz_Li_2006b} sufficient conditions for the existence of  a unique random global attractor are proposed.  This is the main subject that  we will develop in this work.

Indeed, we study the asymptotic behaviour of solutions to the
following problem. Let $\mathcal{O}\subset\mathbb{R}^{2}$ be an
open set, not necessarily bounded, with sufficiently regular boundary
$\partial\mathcal{O}$, and suppose that $\mathcal{O}$ satisfies
the Poincar\'{e} inequality, i.e., there exists a constant $C>0$
 such that
\begin{equation*}
C\int_{\mathcal{O}}\varphi^{2}\,d\xi\leq\int_{\mathcal{O}}|\nabla
\varphi|^{2}\,d\xi\quad\hbox{\rm for all $\varphi\in H^1_0({\mathcal O})$},
\end{equation*}
and  consider  the Navier-Stokes equations (NSE) in $\mathcal{O}$ with
homogeneous Dirichlet boundary conditions:

\begin{equation}
\label{NSE}\left\{
\begin{array}
[c]{l}%
{\dfrac{\partial u}{\partial t}}-\nu\Delta
u+(u\cdot\nabla)u+\nabla
p=f+{\dfrac{dW(t)}{dt}}\text{ \ \ in \ }(0,+\infty)\times \mathcal{O},\\[2ex]%
\old{\nabla\cdot} \new{{\rm div}\, u}\;=0\text{ \ \ in \ }(0,+\infty)\times \mathcal{O},\\[1ex]%
u=0\;\;\mathrm{on}\;(0,+\infty)\times \partial\mathcal{O},\\[1ex]%
u(0)=u_0,\text{ \ \ }\\[1ex]%
\end{array}
\right.
\end{equation}
where $\nu>0$ is the kinematic viscosity, $u$ is the velocity
field of the fluid, $p$ the pressure,  $u_0$ the initial velocity
field, and $f$ a given external force field.  Here $W(t), t \in
\mathbb{R},$ is a two-sided cylindrical Wiener process in
$\mathrm{H}$ with its Reproducing Kernel Hilbert Space (RKHS)
$\mathrm{K}$ satisfying Assumption \ref{ass:A-01} below, defined
on some  probability space $ (\Omega ,{\mathcal F},  \mathbb{P})$.
Note that following \cite{Brz_Li_2006b} we  allow our driving
noise to be much rougher than in previous works in the literature,
see for instance Crauel and Flandoli \cite{Crauel_F_1994} or
Kloeden and Langa \cite{Kloeden+Langa_2007},  for which it is possible to
prove that there exists a random dynamical system associated to
our model. Indeed,
the rougher the noise the closer the model is to reality.
Landau and Lifshitz  in their fundamental 1959 work \cite[Chapter 17]{LL_1987}  proposed   to study NSEs under additional stochastic  small fluctuations. Consequently the authors consider the classical balance laws for mass, energy and momentum forced
by a random noise, to describe the fluctuations, in particular local stresses and temperature,
which are not related to the gradient of the corresponding quantities. In \cite[Chapter 12]{LL_1968} the same authors then derive
correlations for the random forcing by following the general theory of fluctuations. One of the requirements on the noise they impose is that the noise is either spatially uncorrelated or correlated as little as possible. It is known that spatially uncorrelated noise corresponds to the Wiener process with RKHS $L^2$ and if the RKHS of the Wiener process is the Sobolev space $H^{s,2}$, then the smaller the $s$ the less correlated noise is. In other words, the less regular spatially is the less correlated it is. Note that our Wiener process includes a finite dimensional Brownian Motion as a special case.

On the other hand, Caraballo et al. \cite{Caraballo_L_R_2006}
introduced a concept \old{for a cocycle to be \del{pullback}
asymptotically compact} \new{of a   \del{pullback} asymptotically compact
cocycle}, which was successfully used to prove the existence of
\del{pullback} attractors for a 2D non-autonomous Navier-Stokes equations,
and later has been also used to prove existence of random attractors
for stochastic lattice dynamical systems in Bates et al.
\cite{bates2006},  stochastic reaction-diffusion equations in Bates
et al. \cite{bates2009} and a stochastic Benjamin-Bona-Mahony
equation in Wang \cite{wang2009}, all of them related to unbounded
domains. In this paper, we will use the same concept, which
generalizes the one in \cite{Brz_Li_2006b}, to prove the existence and the
uniqueness of global random attractors for our stochastic 2D
Navier-Stokes equations with irregular noise in Poincar\'e unbounded
domains. In this sense, our main result implies that  the stochastic flow associated to our model is \del{pullback}
asymptotically compact (see Proposition \ref{prp:5.1}). It is
remarkable that we also prove the measurability of our random
attractor, which is usually missed in the literature.



\begin{Notation} By $\mathbb{N}$, $\mathbb{Z}$, $\mathbb{Z}^-$, $\mathbb{R}$ and $\mathbb{R}_+$ we will denote respectively  the sets of natural numbers (which includes the zero),   of integers,  of non-positive integers,  of real numbers and of all non-negative real numbers. By $\mathcal{B}(X)$, where $X$ is a topological space,  we will denote the $\sigma$-field of all Borel subsets of $X$.

\end{Notation}
\subsection*{Acknowledgements}
Preliminary versions of this work were presented at Cambridge  and Oberwolafch (May and November 2008). A brief report of the latter lecture is published as \cite{Brz-Oberwolfach}. The research of
the first named author was partially supported  by an EPSRC grant number
EP/E01822X/1.

\section{Stochastic  $2D$-Navier-Stokes equations with additive noise in
unbounded domains}\label{sec-appl-sNSEs}
\subsection{Statement of the problem}
Let $\mathcal{O}\subset\mathbb{R}^{2}$ be an open set, not
necessarily a bounded one. We denote by  $\partial\mathcal{O}$ the  boundary of $\mathcal{O}$. We will always assume that the closure $\bar{\mathcal{O}}$ of the set $\mathcal{O}$ is manifold with boundary of $C^\infty$ class, whose boundary is equal to $\partial\mathcal{O}$, i.e. we will assume that
$\mathcal{O}$ satisfies the condition (7.10) from \cite[chapter I]{L-M-72}:
$$
\leqno(7.10)
\left\{\begin{array}{l}\mbox{ $\partial\mathcal{O}$ is a $1$-dimensional infinitely differentiable manifold, $\mathcal{O}$ being} \\
 \mbox{ locally on one side of $\mathcal{O}$.
}\end{array}\right.
$$
We will also
assume that  $\mathcal{O}$ is a Poincar\'e
domain, i.e.  that there exists a constant $\lambda_1>0$ such that  the following Poincar\'{e} inequality is satisfied
\begin{equation}
\lambda_1\int_{\mathcal{O}}\varphi^{2}\,dx\leq\int_{\mathcal{O}}|\nabla
\varphi|^{2}\,dx\quad\hbox{\rm for all $\varphi\in H^1_0({\mathcal O}).$}
\label{2.0}
\end{equation}

 In order to formulate  our
problem in an abstract framework  let us recall the definitions of  the following usual
functional spaces.
\begin{eqnarray*}
\mathbb{L}^2(\mathcal{O})&=&L^{2}(\mathcal{O},\mathbb{R}^2),\\
\mathbb{H}^k(\mathcal{O})&=&H^{k,2}(\mathcal{O},\mathbb{R}^2),\; k\in\mathbb{N},\\
\mathcal{V}&=&\left\{  u\in
C_{0}^{\infty}(\mathcal{O},\mathbb{R}^2)
;\;\mathrm{div}\,u\;=0\right\} ,\\
H&=& \mbox{the closure of $\mathcal{V}$ in } \mathbb{L}^2(\mathcal{O}),\\
\mathbb{H}_0^1(\mathcal{O})&=& \mbox{the closure of $C_{0}^{\infty}(\mathcal{O},\mathbb{R}^2)
$ in } \mathbb{H}^1(\mathcal{O}),\\
V&=& \mbox{the closure of $\mathcal{V}$ in } \mathbb{H}^1(\mathcal{O}).
\end{eqnarray*}
We endow the set $H$ with the inner product $(\cdot,\cdot)$ and the norm   $\left\vert \cdot\right\vert $ induced by $\mathbb{L}^2(\mathcal{O})$.
Thus, we have
\[
(u,v)=\sum_{j=1}^{2}\int_{\mathcal{O}}u_{j}(x)v_{j}(x)\,{d}x,
\]
Since the set $\mathcal{O}$ is a Poincar\'e domain,  the norms on $V$ induced by $\mathbb{H}^1(\mathcal{O})$ and $\mathbb{H}_0^1(\mathcal{O})$ are equivalent.  The latter norm   and the associated inner
product will be denoted by $\left\Vert \cdot\right\Vert $ and $\big(\!\big(\cdot,\cdot\big)\!\big)$, respectively.  They satisfy the following equality
\[
\big(\!\big(u,v\big)\!\big)=\sum_{i,j=1}^{2}\int_{\mathcal{O}}{\frac{\partial
u_{j}}{\partial x_{i}}\frac{\partial v_{j}}{\partial
x_{i}}}\,{d}x,\; u,v\in\mathbb{H}^1_0(\mathcal{O}).
\]
Since the space $\mathrm{V} $ is  densely and continuously
 embedded into $\mathrm{H}$, by   identifying   $\mathrm{H}$
 with its dual $\mathrm{H}^\prime$, we have the following embeddings
 \begin{equation}
 \label{eqn:Gelfanf}
\mathrm{V} \subset \mathrm{H}\cong\mathrm{H}^\prime \subset
\mathrm{V}^\prime.\end{equation}
Let us observe here  that, in particular,  the spaces $\mathrm{V}$, $\mathrm{H}$ and
$\mathrm{V}^\prime$ form a Gelfand triple.

We will denote by  $| \cdot |_{V^{\prime}}$   and  $\left\langle \cdot,\cdot\right\rangle $ the norm in
$V^{\prime}$ and  the
duality pairing between $V$ and $V^{\prime}$, respectively.

The presentation of the Stokes operator is standard and we  follow here the one given in  \cite{Brz_Li_2006b}.
We  begin with  defining  a bilinear
form $a:\mathrm{V}\times \mathrm{V} \to \mathbb{R}$ by
 \begin{equation}
 \label{form-a}
a(u,v):=(\nabla u,\nabla v), \quad u,v \in \mathrm{V}.
\end{equation}
Since  obviously the form   $a$ coincides with   the  $ \vnl
\cdot, \cdot\vnr$ scalar product in $\mathrm{V}$, it is
$\mathrm{V}$-continuous, i.e. it satisfies $\vert a(u,u) \vert
\leq C \Vert u \Vert^2$ for some $C
>0$ and all $u \in \mathrm{V}$. Hence, by the Riesz Lemma,
 there exists a unique linear operator
$\mathcal{A}:\mathrm{V} \to \mathrm{V}^\prime$,  such that
$a(u,v)=\lb \mathcal{A}u,v\rb$, for $u, v \in \mathrm{V}$. Moreover, since the $\mathcal{O}$ is a Poincar\'e
domain,  the form $a$ is  $\mathrm{V}$-coercive, i.e. it satisfies
$a(u,u) \geq \alpha \Vert u \Vert^2$ for some $\alpha
>0$ and all $u \in \mathrm{V}$. Therefore, in view  of the Lax-Milgram theorem, see for instance
   Temam \cite[Theorem II.2.1]{Temam-97},
 the operator $\mathcal{A}:\mathrm{V} \to \mathrm{V}^\prime$ is an
 isomorphism.

 Next we  define an unbounded linear operator
$\mathrm{A}$ in $\mathrm{H}$ as follows.
\begin{equation}
\label{def-A} \left\{
\begin{array}{ll}
D(\mathrm{A}) &:= \{u \in \mathrm{V}: \mathcal{A}u \in
\mathrm{H}\}\cr \mathrm{A}u&:=\mathcal{A}u, \, u \in
D(\mathrm{A}).
\end{array}
\right.
\end{equation}

It is now well established that under some  assumptions\footnote{These assumptions are satisfied in our case}
related to  the regularity of the domain $\mathcal{O}$, the space
$D(\mathrm{A})$ can be characterized in terms of  the Sobolev spaces.
For example, see \cite{Hey-80}, where only the 2-dimensional case
is studied but the result is also valid in the 3-dimensional case,
if $\mathcal{O}\subset \mathbb{R}^2$ is a uniform $C^2$-class
Poincar\'{e} domain, then with
$\mathrm{P}:\mathbb{L}^2(\mathcal{O}) \rightarrow \mathrm{H}$
being the orthogonal projection, we have
\begin{equation}
\label{eqn:4.3} \left\{
\begin{array}{ll}
D(\mathrm{A}) &:= \mathrm{V} \cap \mathbb{H}^2(\mathcal{O}),\cr
\mathrm{A}u&:=-\mathrm{P}\Delta u, \quad u\in D(\mathrm{A}).
\end{array}
\right.
\end{equation}
It is also a classical result,  see e.g. Cattabriga \cite{Cattabriga_1961} or Temam
\cite{Temam-97}, p. 56, that $\mathrm{A}$ is a non-negative self adjoint operator in
$\mathrm{H}$.  Moreover, see  p. 57 in \cite{Temam-97}, $\mathrm{V}=D(\mathrm{A}^{1/2})$.  Let us  recall a result of
Fujiwara--Morimoto \cite{Fuj+Mor_1977} that the projection $\mathrm{P}$ extends to a
bounded linear projection in the space $\mathbb{L}^q(D)$, $1< q< \infty$.

Consider the trilinear form $b$ on $V\times V\times V$ given by
\[
b(u,v,w)=\sum_{i,j=1}^{2}\int_{\mathcal{O}}u_{i}{\frac{\partial v_{j}%
}{\partial x_{i}}}w_{j}\,\,{d}x,\quad u,v,w\in V.
\]
Indeed, $b$ is a continuous bilinear form and, see for instance \cite{Temam-84}, Lemma 1.3, p.163 and Temam \cite{Temam-97},
\begin{equation}
\begin{aligned}
\label{eqn:b01} &b(u,v,v)=0,  \quad \mbox{\rm  for} \, u\in
\mathrm{V}, v\in \mathbb{H}_0^{1,2}(\mathcal{O}), \cr &b(u,v,w)=-b(u,w,v),
\quad \mbox{\rm for} \, u\in \mathrm{V}, v, w\in
\mathbb{H}_0^{1,2}(\mathcal{O}),
\end{aligned}
\end{equation}
\begin{equation}
\begin{aligned}
\label{eqn:4.0a}
 \vert b(u,v,w)  \vert \leq C\left\{
\begin{array}{ll}
 \vert u\vert ^{1/2}\vert \nabla u\vert^{1/2}\vert \nabla  v\vert ^{1/2}\vert \mathrm{A}v\vert ^{1/2}\vert w\vert ,
\quad \, u \in \mathrm{V}, v\in D(\mathrm{A}), w\in \mathrm{H},\\
 \vert u\vert ^{1/2}\vert \mathrm{A}u\vert ^{1/2}\vert \nabla  v\vert \vert w\vert , \quad
 u \in D(\mathrm{A}), v\in \mathrm{V}, w\in \mathrm{H},\\
 \vert u\vert \vert \nabla  v\vert \vert w\vert ^{1/2}\vert \mathrm{A}w\vert ^{1/2}, \quad
 u \in \mathrm{H}, v\in \mathrm{V}, w\in D(\mathrm{A}),\\
 \vert u\vert ^{1/2}\vert \nabla  u\vert ^{1/2}\vert \nabla  v\vert \vert w\vert ^{1/2}\vert \nabla  w\vert ^{1/2}, \quad
 u, v, w \in \mathrm{V}
\end{array}
\right.
\end{aligned}
\end{equation}
for some $C>0.$
Define next a bilinear map $B:V\times V\rightarrow V^{\prime}$ by
\[
\left\langle B(u,v),w\right\rangle =b(u,v,w),\quad u,v,w\in V,
\]
and a homogenous polynomial of second degree $B:V \rightarrow V^{\prime}$ by

\[
B(u)=B(u,u),\; u\in V.
\]

Let us also recall \cite[Lemma 4.2]{Brz_Li_2006b}.
\begin{lemma}\label{lem:form-b} The trilinear map
$b:\mathrm{V}\times \mathrm{V} \times \mathrm{V} \to \mathbb{R}$
has a unique extension to a bounded trilinear map from
$\mathbb{L}^4(\mathcal{O}) \times (\mathbb{L}^4(\mathcal{O})
\cap\mathrm{H})\times \mathrm{V}$ and from
$\mathbb{L}^4(\mathcal{O}) \times \mathrm{V}\times
\mathbb{L}^4(\mathcal{O})$ to $\mathbb{R}$. Moreover, $B$ maps
$\mathbb{L}^4(\mathcal{O})\cap\mathrm{H}$ (and so $\mathrm{V}$)
into $\mathrm{V}^\prime$ and
\begin{equation}\label{eqn:4.0}
\Vert B(u) \Vert_{\mathrm{V}^\prime} \leq C_1\vert u
\vert^2_{\mathbb{L}^4(\mathcal{O})} \leq 2^{1/2}C_1 \vert u \vert  \vert
\nabla  u \vert \leq C_2\vert u \vert_V^2 , \quad
 u \in \mathrm{V}.
\end{equation}
\end{lemma}
\begin{proof} It it enough to observe
that from the H\"older inequality we  have  the following
inequality
\begin{equation}\label{eqn:4.00}
\vert b(u,v, w)\vert \le C\vert u\vert_{\mathbb{L}^4(\mathcal{O})} \vert
\nabla v \vert _{\mathbb{L}^2(\mathcal{O})}  \vert w
\vert_{\mathbb{L}^4(\mathcal{O})}, \quad u, v, w \in \mathbb{H}_0^{1,2}(\mathcal{O}).
\end{equation}
\end{proof}

\subsection{ Attractors for random dynamical systems}
\label{subsec-2.2}

\begin{definition}
\label{Def-DS} A triple $ \mathfrak{T}=\left(\Omega,\mathcal{F},
\vartheta\right) $ is called a  \textbf{measurable dynamical
system} (DS) iff
 $(\Omega,\mathcal{F})$ is a measurable space and
 $\vartheta:\mathbb{R}\times\Omega\ni (t,\omega)\mapsto
\vartheta_{t}\omega \in \Omega $ is a measurable map such that
$\vartheta_{0}=$identity, and
 for all $t,s \in
\mathbb{R}$, $\vartheta_{t+s}=\vartheta_{t}\circ\vartheta_{s}$. \\
A quadruple $ \mathfrak{T}=\left(\Omega,\mathcal{F}, \mathbb{P},
\vartheta\right) $ is called \textbf{metric DS}  iff
$(\Omega,\mathcal{F}, \mathbb{P})$ is a  probability space and
$\mathfrak{T}^\prime:=\left(\Omega,\mathcal{F}, \vartheta\right) $
is  a measurable DS such that for each $t\in \mathbb{R}$, the map
$\vartheta_{t}:\Omega \to \Omega$ is $\mathbb{P}$-preserving.

\end{definition}

\begin{definition}
\label{Def-RDS}  Suppose  that $\mathrm{X}$ is  a Polish space,
i.e. a metrizable complete separable  topological space,
$\mathcal{B}$ is its Borel $\sigma-$field and
 $\mathfrak{T}$ is  a metric DS.  A
map $\varphi: \mathbb{R}_+\times\Omega\times\mathrm{X} \ni
(t,\omega,x) \mapsto \varphi(t,\omega,x) \in \mathrm{X}$
  is called a \textbf{measurable random dynamical system} (RDS) (on $\mathrm{X}$ over $\mathfrak{T}$), iff

\begin{enumerate}
\item[(i)]  $\varphi$ is
$\left(\mathcal{B}(\mathbb{R}_+)\bigotimes\mathcal{F}\bigotimes\mathcal{B},
\mathcal{B}\right)$-measurable;
 \item[(ii)] $\varphi$ is a $\vartheta$-cocycle, i.e.
 $$\varphi(t+s,\omega,x)=\varphi\big(t,\vartheta_s\omega,\varphi(s,\omega,x)\big).$$

\end{enumerate}
The map $\varphi$ is said to be
\textbf{continuous} iff, for all $(t,\omega) \in
\mathbb{R}_+\times \Omega$, $\varphi(t, \omega,\cdot): \mathrm{X}
\rightarrow \mathrm{X}$ is continuous. Similarly,
$\varphi$ is said to be time continuous iff,
for all $\omega \in \Omega$ and $x\in \mathrm{X}$, the map
$\varphi(\cdot,\omega,x):\mathbb{R}_+ \rightarrow \mathrm{X}$ is
continuous.
\end{definition}


The notion of a random set is presented following \cite{Brz_Cap_Fland_1993}, see also Crauel \cite{Crauel_1995} and Definition 2.3 in \cite{Brz_Li_2006b}.
For two non-empty sets $A, B \subset X$, where  $(\mathrm{X},d)$ is a Polish space, we  put $$d(A,B)=\sup_{x\in
A}d(x,B)\quad \hbox{ and }\quad \rho(A,B)= \max\{d(A,B),d(B,A)\}.$$
The latter is called the Hausdorff distance.
 It is known that   $\rho$  restricted to the family $\mathfrak{CB}(X)$ (the family of all non-empty closed and bounded
subsets of $\mathrm{X}$)  is a distance, see  Castaing and Valadier
\cite{C+V_1977}. From now on, let $\mathfrak{X}$ be the $
\sigma-$field on $\mathfrak{CB}$ generated by open sets with
respect to  the Hausdorff metric $\rho$, e.g.
\cite{Brz_Cap_Fland_1993},  \cite{C+V_1977} or Crauel
\cite{Crauel_1995}.

\begin{definition}\label{Def-randomset} Let us assume that $(\Omega,\mathcal{F})$ is a measurable
space and $(\mathrm{X},d)$ is a Polish space.  A set
valued map $C:\Omega\rightarrow \mathfrak{CB}(X)$  is said to be measurable iff $ C $ is
$\left(\mathcal{F}, \mathcal{X}\right)$-measurable. Such a map $C$ will  often be called  a \textbf{ closed and bounded random
set} on $X$. A closed and bounded random  set $C$ on $X$ will be called a
\textbf{compact random   set} on $X$ iff for each
$\omega\in\Omega$, $C(\omega)$ is a compact   subset of $X$.

\end{definition}

\begin{remark}\label{rem-chueshov}{\rm
Let $f:X\mapsto\mathbb{R}_+$ be a continuous function on the
Polish space $X$, and $R:\Omega\mapsto\mathbb{R}_+$  an
$\mathcal{F}$-measurable random variable. If the set
$C_{f,R}(\omega):=\{x:\;f(x)\leq R(\omega)\}$ is non-empty for
each $\omega\in\Omega,$ then $C_{f,R}$ is a closed and bounded
random set (see \cite{Chueshov-random} Proposition 1.3.6).
}\end{remark}

We  denote by $\mathcal{F}^u$ the $\sigma$-algebra of universally measurable sets associated to
the measurable space $(\Omega,\mathcal{F})$, see Crauel's monograph \cite{Crauel_1995} for the definition and basic properties. 

To our best knowledge, the following definition appeared for the first time in the fundamental work by Fladoli and Schmalfuss \cite{Flandoli-Schmalfuss_1996}, see Definition 3.4. 

\begin{definition}
\label{randomattractor} A random set $A:\Omega\rightarrow\mathfrak{CB}(X)$  is
a \textbf{random  $\mathfrak{D}$-attractor} iff
\begin{trivlist}
\item[(i)] $A$ is a
compact random set,
\item[(ii)] ${A}$ is  $\varphi$-invariant, i.e., P-a.s.
$$
\varphi (t,\omega)A(\omega)=A(\vartheta _t\omega)
$$
\item[(iii)]   $A$  is
 $\mathfrak{D}$-attracting, in the sense that, for all $D\in \mathfrak{D}$ it holds
 $$
 \lim_{t\to \infty}d(\varphi (t, \vartheta _{-t}\omega)D(\vartheta _{-t}\omega), A(\omega))=0.
 $$
\end{trivlist}
\end{definition}

\begin{definition}\label{DEF-Asymptotic compact}
 We say that a  RDS $\vartheta$-cocycle $\varphi$ on $X$
  is $\mathfrak{D}$-\textbf{asymptotically compact} iff for each $D\in\mathfrak{D}$, for every
$\omega \in \Omega$,  for  any positive sequence $(t_n)$ such that
$t_n \rightarrow \infty$ and for any  sequence $\{x_n\} $ such that
$$x_{n}\in D(\vartheta_{-t_{n}}\mathbb{\omega}), \; \mbox{ for all }
n\in\mathbb{N},$$
the following set is  pre-compact in
$\mathrm{X}$:
$$\{\varphi(t_n,\vartheta
_{-t_n}\omega,x_n): n \in \mathbb{N}\}.$$
\end{definition}

We now write the result \del{we need} on the
existence of a random $\mathfrak{D}$-attractor, see \cite{Caraballo_L_R_2006, bates2006, bates2009, wang2009}.

\begin{theorem}
\label{Teorema1} Assume that
$\mathfrak{T}=\left(\Omega,\mathcal{F}, \mathbb{P},
\vartheta\right)$ is a metric DS, $\mathrm{X}$ is a Polish  space,
$\mathfrak{D}$ is a nonempty class of closed and bounded  random sets on $X$ and
$\varphi$ is a continuous, $\mathfrak{D}$-asymptotically compact
RDS on $\mathrm{X}$ (over $\mathfrak{T}$). Assume that there exists a
$\mathfrak{D}$-absorbing closed and bounded  random set $ {B}$ on $X$, i.e., given $D\in \mathfrak{D}$ there exists $t(D))$ such that
$\varphi (t,\vartheta _{t}\omega)D(\vartheta _{-t}\omega)\subset B(\omega)$ for all $t\geq t(D).$
Then, there exits an   $\mathcal{F}^u$-measurable $\mathfrak{D}$-attractor $A$ given  by
\begin{equation}
A(\mathbb{\omega})=\Omega_{B}(\mathbb{\omega}),\quad\mathbb{\omega
}\in\Omega, \label{A}
\end{equation}
with
$$
\Omega_B(\omega)=\bigcap_{T\geq0}\overline{\bigcup_{t\geq
T}\varphi(t, \vartheta_{-t} \omega,B(\vartheta_{-t}
\omega)}),\quad\omega\in\Omega.
$$
\end{theorem}

\textcolor{blue}{\begin{remark}\label{rem-Kloeden+Langa_2007}{\rm
One should mention here  that a related paper \cite{Kloeden+Langa_2007}
is about   2D Navier-Stokes in \textbf{bounded} domains
and and with much more regular noise and its results do not imply those from the current paper.   On the other hand Theorem 4.6 from that paper
is  applicable to Stochastic NSEs in 2-D unbounded
domains, instead of Theorem \ref{Teorema1} provided  one can prove that the corresponding system is asymptotically compact  and
has random bounded absorbing set.  That's is what we do in our
paper, with a small but important difference that our class of families of random sets with respect
to which AC and absorption hold are different. Our Theorem \ref{Teorema1} on the existence of attractor is
generalization (or modification, if one prefers) of the above Theorem 4.6 to the case considered in the present paper.
}\end{remark}}

\begin{proof} The existence of $A(\omega)$ follows from Theorem 7 in \cite{Caraballo_L_R_2006}. We only need to prove the measurability claim. For this we will follow a slight modification of the proof of Proposition 1.6.2
in \cite{Chueshov-random}, see also the proof of Lemma 2.3 in \cite{Brz_Cap_Fland_1993}. Observe that evidently, for every $\omega\in\Omega,$
$$\Omega_{B}(\mathbb{\omega})=\bigcap_{n\in\mathbb{Z}_+}\overline{\gamma_B^n(\omega)},$$
where, by definition,
$$\gamma_B^n(\omega)=\bigcup_{t\geq
n}\varphi\big(t, \vartheta_{-t} \omega,B(\vartheta_{-t}
\omega)\big).$$

Let us fix $x\in X$. Since $\gamma_B^n(\omega)\subset\gamma_B^{n+1}(\omega)\subset\overline{\gamma_B^{n+1}(\omega)}\subset\overline{\gamma_B^n(\omega)},$
we have that
$$d\big(x,\gamma_B^{n}(\omega)\big)\leq d\big(x,\gamma_B^{n+1}(\omega)\big)\leq d\big(x,\Omega_B(\omega)\big),$$
and therefore there exists the $\displaystyle\lim_{n\to\infty}d\big(x,\gamma_B^{n}(\omega)\big),$
and $$\lim_{n\to\infty}d\big(x,\gamma_B^{n}(\omega)\big)\leq d\big(x,\Omega_B(\omega)\big),\quad\forall\,\omega\in\Omega.$$

Let us fix $\omega\in\Omega,$ and take $x_n\in\gamma_B^{n}(\omega)$ such that

$$d(x,x_n)\leq d\big(x,\gamma_B^{n}(\omega)\big)+\displaystyle\frac{1}{n},\quad n\in\mathbb{N}.$$
Since $\varphi$ is a  $\mathfrak{D}$-asymptotically compact
RDS on $\mathrm{X}$, there exists a subsequence $n_k=n_k(\omega)$ and an element $y=y(\omega)$ such that
$x_{n_k}\to y$. Evidently, $y\in\Omega_B(\omega)$\del{(see Remark \ref{rem:2.0})}. Therefore,
$$d\big(x,\Omega_B(\omega)\big)\leq d(x,y)=\lim_{k\to\infty}d(x,x_{n_k})\leq\lim_{n\to\infty}d\big(x,\gamma_B^{n}(\omega)\big).$$

Thus, we get

$$d\big(x,\Omega_B(\omega)\big)=\lim_{n\to\infty}d\big(x,\gamma_B^{n}(\omega)\big)\quad\forall\,\omega\in\Omega,$$
and consequently, observing that by Proposition 1.5.1 in
\cite{Chueshov-random} the map $\omega\mapsto d\big(x,\gamma_B^{n}(\omega)\big)$ is $\mathcal{F}^u$-measurable,
we obtain that
the map $\omega\mapsto d\big(x,\Omega_B(\omega)\big)$ is also
$\mathcal{F}^u$-measurable.
\end{proof}

\begin{remark}\label{RACDFremark}{\rm
 If $\mathfrak{D}$
contains every bounded and closed nonempty deterministic subsets
of $X$, then
as a consequence of this theorem, of Theorem 2.1 in \cite{Crauel4}, and of Corollary 5.8 in
\cite{Crauel_F_1999},
we obtain that the random attractor $A$ is given by
\begin{equation}\label{RACDF}
A(\omega)=\overline{\bigcup_{C\subset X}\Omega_C(\omega)}\quad\mathbb{P}-a.s.,
\end{equation}where the union in \eqref{RACDF}
is made for all bounded and closed nonempty deterministic subsets $C$
of $X$.
}\end{remark}

\subsection{Stochastic Navier-Stokes equations with an additive noise}
\label{subsec-3.1}

The model we consider in this subsection is the same as  the one studied in \cite{Brz_Li_2006b}. It was shown therein that the
RDS generated by the stochastic NSEs  below is asymptotically
compact. We will strengthen that result by showing that
\begin{trivlist}
\item[(i)] it is $\mathfrak{D}$-asymptotically
compact  and \item[(ii)] there exists a family
${B}\in\mathfrak{D}$ which is
$\mathfrak{D}$-absorbing,
\end{trivlist} for a family $\mathfrak{D}$ of random closed and bounded  sets to be defined below.
Thus we will conclude that Theorem \ref{Teorema1} is applicable.

Our aim in this subsection is to study the following stochastic
Navier-Stokes equations in $\mathcal{O}$
\begin{equation}\label{eqn_SNSE_add_noise}
\left\{
\begin{array}{ll}
du + \{\nu \mathrm{A}u + B(u) \}\,dt= f\,dt+dW(t), \quad t \geq 0\\
u(0)=x ,
\end{array}\right.
\end {equation}
where we assume that  $x\in \mathrm{H}$,  $f \in
\mathrm{V}^\prime$ and $W(t), t \in \mathbb{R},$ is a two-sided
cylindrical Wiener process in $\mathrm{H}$ with its Reproducing
Kernel Hilbert Space (RKHS) $\mathrm{K}$ satisfying Assumption
\ref{ass:A-01} below (see Remark 6.1 in \cite{Brz_Li_2006b})
defined on some  probability space $ (\Omega ,{\mathcal F},
\mathbb{P})$. The following is our standing assumption.

\begin{Assumption}\label{ass:A-01}
$\mathrm{K} \subset \mathrm{H} \cap \mathbb{L}^4(\mathcal{O})$ is a Hilbert
space such that
 for some $\delta \in (0,1/2)$,
\begin{equation}
\label{A-gamma}
\begin{aligned} \mathrm{A}^{-\delta}:K \to \mathrm{H} \cap \mathbb{L}^4(\mathcal{O})
\hbox{ is } \gamma\hbox{-radonifying}.
\end{aligned}
\end{equation}
\end{Assumption}

Let us denote $\mathrm{X}=\mathrm{H} \cap
\mathbb{L}^4(\mathcal{O})$ and let $\mathrm{E}$ be the completion
of $X$ $A^{-\delta}(\mathrm{X})$ with respect to the image norm
$\vert x \vert_{\mathrm{E}}=\vert A^{-\delta}x\vert_{\mathrm{X}}$,
$x \in \mathrm{X}$.
 It is  well known that $\mathrm{E}$ is a
separable Banach space. For  $\xi \in[0,1/2)$ we set
$$
\Vert \omega \Vert_{\Cer}=\sup_{t \not= s \in \mathbb{R}}\frac{\vert
\omega(t)-\omega(s)\vert_\mathrm{E}}{\vert t-s \vert^\xi (1+\vert
t \vert + \vert s \vert)^{1/2}}.
$$
By $C_{1/2}^\xi(\mathbb{R},\mathrm{E})$ we will denote the set of all
$\omega \in C(\mathbb{R},\mathrm{E})$ such that   $\omega(0)=0$ and $\Vert \omega \Vert_{\Cer}<\infty$.
It is easy  to prove that the  closure of $\{
\omega \in C_0^\infty(\mathbb{R},\mathrm{E}): \omega(0)=0\}$ in
$\Cer$, denoted by $\Omega(\xi,\mathrm{E})$,  is a separable
Banach space\footnote{But for $\xi\in(0,1)$  $\Cer$  endowed with the norm
$\Vert \cdot \Vert_{\Cer}$
is   not separable.}.

Finally, we set
$$
\Vert\omega \Vert_{C_{1/2}(\mathbb{R},\mathrm{E})}=\sup_{t\in
\mathbb{R}}\frac{\vert \omega(t)\vert_\mathrm{E}}{ 1+\vert t
\vert^{1/2}}
$$
and \del{we} denote  by
$C_{1/2}(\mathbb{R},\mathrm{E})$  the space of all continuous
functions  $\omega: \mathbb{R} \to \mathrm{E}$ such that $\Vert\omega \Vert_{C_{1/2}(\mathbb{R},\mathrm{E})}<\infty$.
Then the space $C_{1/2}(\mathbb{R},\mathrm{E})$ endowed with the  norm $\Vert\cdot \Vert_{C_{1/2}(\mathbb{R},\mathrm{E})}$
is a separable Banach space.

 By  $\mathcal{F}$  we will denote the Borel $\sigma$-algebra on $\Oer$.
One can show by methods from \cite{Brz_1996}, but see also
\cite{Hairer_2003} for a similar problem in the one dimensional
case, that for $\xi \in (0, 1/2)$, there exists a Borel
probability measure $\mathbb{P}$ on $\Oer$ such that the canonical
process $w=(w_t)_{t\in \mathbb{R}}$, defined by
\begin{equation}
\label{def:BM} w_t(\omega):=\omega(t), \quad \omega \in \Oer, \;t\in \mathbb{R},
\end{equation}
where
$i_t: \Oer \ni \gamma \mapsto \gamma(t) \in \mathrm{E}$,  is the evaluation map at time $t$,
is a two-sided Wiener process such that the Cameron-Martin, i.e. the
Reproducing Kernel Hilbert,  space of the Gaussian measure
$\mathcal{L}(w_1)$ on $\mathrm{E}$ is equal to $\mathrm{K}$.

For
$t \in \mathbb{R}$, let $\mathcal{F}_t:=\sigma\{w_s: s \leq t\}$.
Since for each $t \in \mathbb{R}$ the map $z\circ i_t:
\mathrm{E}^\ast \to L^2(\Oer,\mathcal{F}_t,\mathbb{P}) $  satisfies
$\mathbb{E} \vert z\circ i_t \vert^2 =t \vert z
\vert_{\mathrm{K}}^2$, there exists a unique extension of $z\circ
i_t$ to a bounded linear map $W_t:\mathrm{K} \to
L^2(\Oer,\mathcal{F}_t,\mathbb{P}) $. Moreover, the family
$(W_t)_{t \in \mathbb{R}}$ is a $\mathrm{H}$-cylindrical Wiener
process on a filtered probability space $(\Oer,
(\mathfrak{F}_t)_{t \in \mathbb{R}}, \mathbb{P})$ in the sense of
definition given in  \cite{Brz-P-01}.

On the space $C_{1/2}(\mathbb{R},\mathrm{X})$
 we consider  a flow $\vartheta=\left(
\vartheta_t\right)_{t\in \mathbb{R}}$  defined by
$$
\vartheta_t \omega(\cdot)=\omega(\cdot+t)-\omega(t), \quad \omega
\in \Omega, \; t \in \mathbb{R}.
$$
With respect to this flow  the spaces  $\Cer$ and $\Omega(\xi, \mathrm{E})$
are invariant and we will often denote by $\vartheta_t$ the
restriction of $\vartheta_t$ to any one of these spaces.

It is obvious that for each $t\in \mathbb{R}$, $\vartheta_t$
preserves $\mathbb{P}$. In order to define an Ornstein-Uhlenbeck
process we need to recall some analytic preliminaries from
\cite{Brz_Li_2006b}.

\begin{proposition}\label{prp:z-1}
Assume that $A$ is a generator of an analytic semigroup
$\{e^{-tA}\}_{t \geq 0}$ on a separable Banach space $\mathrm{X}$,
such that for some $C>0$ and $\gamma>0$
\begin{equation}\label{eqn:A-1} \Vertt A^{1+\delta}e^{-tA} \Vertt _{\mathcal{L}(\mathrm{X}, \mathrm{X})} \leq
C t^{-1-\delta}e^{-\gamma t}, \quad t \geq 0.
\end{equation}

For $\xi\in  (\delta, \frac{1}{2})$ there exists a unique linear and
bounded map $\hat{z}: \Cx \to
C_{1/2}(\mathbb{R},\mathrm{X})$ such that for any  $\tilde{\omega}\in \Cx$
\begin{equation}\label{eqn:A-def-z}
 \hat{ z}(t)=\hat{ z}(\tilde{\omega})(t)=\int_{-\infty}^t
A^{1+\delta}e^{-(t-r)A}\big(\tilde{\omega}(t)-\tilde{\omega}(r)\big)\, dr,
t \in \mathbb{R}.
\end{equation}
  In particular,
 there exists a constant $C_2< \infty$ such that for any $\tilde{\omega}\in
\Cx$
\begin{equation}\label{eqn:A-sublinear}
\vert \hat{ z}(\tilde{\omega})(t)\vert_\mathrm{X} \leq C_2(1+\vert t
\vert^{1/2})\Vert \tilde{\omega} \Vert, \quad t \in \mathbb{R}.
\end{equation}
Moreover,  the above results are valid with the space $\Cx$
replaced by $\Omega(\xi, \mathrm{X})$.
\end{proposition}
\begin{proof} See Proposition 6.2 in \cite{Brz_Li_2006b}. The last part follows as $\Omega(\xi, \mathrm{X})$ is a closed
subset of $\Cx$. \end{proof}

The following results are respectively Corollary 6.4, Theorem 6.6
and Corollary 6.8 from \cite{Brz_Li_2006b}.

\begin{corollary}\label{cor:z-2}
Under the assumptions of Proposition \ref{prp:z-1}, for all $-\infty<a <b < \infty$ and $t\in\mathbb{R}$, the map
\begin{equation}
\label{eqn:z-continuous-trajectories}  \Cx \ni \tilde{\omega}
\mapsto \left( {\hat z}(\tilde{\omega})(t), {\hat
z}(\tilde{\omega})\right)  \in \mathrm{X}\times  L^4(a,
b;\mathrm{X})
\end{equation}
is continuous.
Moreover, the above result is valid with the  space $\Cx$ being
replaced by $\Omega(\xi,\mathrm{X})$.
\end{corollary}

\begin{theorem}\label{thm:z-theta}
Under the assumptions of Proposition \ref{prp:z-1}, for any
$\omega \in \Cx$,
\begin{equation}
\label{eqn:cocycle} {\hat z}(\vartheta_s\omega)(t)={\hat
z}(\omega)(t+s), \quad t, s \in \mathbb{R}.
\end{equation}
In particular, for any $\omega \in \Omega$ and all $t, s \in
\mathbb{R}$, ${\hat z}(\vartheta_s\omega)(0)={\hat z}(\omega)(s)$.
\end{theorem}

 For $\zeta\in C_{1/2}(\mathbb{R},\mathrm{X})$ we put%
$$(\tau_s \zeta)=\zeta(t+s), \quad t,s \in \mathbb{R}.
$$
Thus, $\tau_s$ is a a linear and bounded map from
$C_{1/2}(\mathbb{R},\mathrm{X})$ into itself. Moreover, the family
$(\tau_s)_{s \in \mathbb{R}}$ is a $C_0$ group  on
$C_{1/2}(\mathbb{R},\mathrm{X})$.

Using this notation   Theorem \ref{thm:z-theta} can be rewritten
in the following way.
\begin{corollary}\label{cor:tau-theta} For $s \in \mathbb{R}$ $\tau_s\circ {\hat z}={\hat z}\circ \vartheta_s$, i.e.
$$ \tau_s\big({\hat z}(\omega)\big)={\hat z}\big(\vartheta_s(\omega)\big), \;\omega \in \Cx.$$
\end{corollary}

Note that   for any $\nu >0$ and $\alpha \geq 0$, $(\nu A+\alpha
I)^\delta: \mathrm{E} \to \mathrm{X}$ is a bounded linear map and
so is the induced map $\Oe \ni \omega \mapsto (\nu A+\alpha
I)^\delta \omega \in \Ox$.

For $\delta$ as in the Assumption \ref{ass:A-01}, $\alpha \geq 0$,
$\nu >0$, $\xi \in (\delta,1/2)$ and $\omega \in \Ce$ (so that
$(\nu A+\alpha I)^{-\delta}\omega
 \in \Cx$),  we put
 $z_\alpha(\omega):={\hat z}\big((\nu A+\alpha I)^{-\delta}\omega\big)\in
C_{1/2}(\mathbb{R},\mathrm{X})$. Hence,  for any $t\geq 0$,
\begin{eqnarray}
z_\alpha(\omega)(t)&:=&\int_{-\infty}^t(\nu A+\alpha I)^{1+\delta}
e^{-(t-r)(\nu A+\alpha I)} \nonumber\\
&& \hspace{2truecm} \left[ (\nu A+\alpha
I)^{-\delta}\omega(t)-(\nu A+\alpha I)^{-\delta}\omega (r)\right]
\, dr
\label{eqn-z_alpha}\\
&=&\int_{-\infty}^t (\nu A+\alpha I)^{1+\delta} e^{-(t-r)(\nu
A+\alpha I)}\big((\nu A+\alpha I)^{-\delta}\vartheta_r\omega\big)(t-r) \,
dr. \nonumber
\end{eqnarray}

It follows  from Theorem \ref{thm:z-theta} that
\begin{equation}\label{eqn_z-stat}
z_\alpha(\vartheta_s\omega)(t)=z_\alpha(\omega)(t+s), \quad \omega
\in \Cx,\;t, s \in \mathbb{R}.
\end{equation}

Let us make the following crucial observation, see Proposition 6.10 in
\cite{Brz_Li_2006b}: the process $z_\alpha$ is an $X$-valued  stationary and
ergodic. Hence,  by the Strong Law of Large Numbers,
see  \cite{DaPZ-96} for a similar
argument,
\begin{equation}
\label{eqn:slln} \lim_{t\rightarrow
\infty}\frac{1}{t}\int_{-t}^{0}\vert
z_\alpha(s)\vert_{\mathrm{X}}^2 \,ds =\mathbb{E}\vert z_\alpha(0)
\vert_{\mathrm{X}}^2, \quad \mathbb{P}-\hbox{a.s. on }\Cx. \end{equation}

Therefore, it follows from \cite[Proposition 6.10]{Brz_Li_2006b} that we
find $\alpha_0$ such that for all $\alpha\geq \alpha_0$,
\begin{equation}\label{universalC}
\mathbb{E}\vert z_\alpha(0)
\vert_{\mathrm{X}} ^2 < \frac{\nu^2\lambda_1}{6C^2},
\end{equation}
where $\lambda_1$ is the constant appearing in the Poincar{\'e} inequality  \eqref{2.0} and
$C>0$ is a certain universal constant.

By $\Omega_\alpha(\xi,\mathrm{E})$ we denote the set of those
$\omega \in \Omega(\xi,\mathrm{E})$ for which the equality
\eqref{eqn:slln} holds true. As mentioned above, the set
$\Omega_\alpha(\xi,\mathrm{E})$ is $\mathbb{P}$-conegligible. Moreover, in view of
 Corollary \ref{cor:tau-theta} that the set $\Omega_\alpha(\xi,\mathrm{E})$  is
invariant with respect to the flow $\vartheta$, i.e. for all
$\alpha \geq 0$ and all $t \in \mathbb{R}$,
$\vartheta_t\big(\Omega_\alpha(\xi,\mathrm{E})\big) \subset
\Omega_\alpha(\xi,\mathrm{E})$. Therefore, we fix
\begin{equation*}\label{cond-xi}\xi \in
(\delta,\frac12)\end{equation*}
 and set
\begin{equation*}\label{def-Omegahat}\Omega:= {\hat \Omega}(\xi,\mathrm{E}) = \bigcap_{n=1}^\infty \Omega_n(\xi,\mathrm{E}).
\end{equation*}
For reasons that will become clear later we
 take   as a model of a  metric DS the  quadruple
$\left(\hat{\Omega}(\xi,\mathrm{E}),\hat{\mathcal{F}},\hat{\mathbb{P}},\hat{\vartheta}\right)$
where $\hat{\mathcal{F}}$, $\hat{\mathbb{P}}$ and
$\hat{\vartheta}$ are respectively the natural restrictions of
$\mathcal{F}$, $\mathbb{P}$ and $\vartheta$ to
$\hat{\Omega}(\xi,\mathrm{E})$.

\begin{proposition}
\label{prop:metricDS} The quadruple
$\left(\hat{\Omega}(\xi,\mathrm{E}),\hat{\mathcal{F}},\hat{\mathbb{P}},\hat{\vartheta}\right)$
is a  metric DS.  For each $\omega \in
\hat{\Omega}(\xi,\mathrm{E})$ the limit in \eqref{eqn:slln}
exists. \end{proposition}
Let us now formulate an immediate and important consequence of the
above result in which  $C>0$ is the constant appearing in \eqref{universalC}.

\begin{corollary}\label{cor-ergodicestimates}
For each $\omega \in \hat{\Omega}(\xi,\mathrm{E})$ there exits
$t_0=t_o(\omega)\geq 0$, such that
\begin{equation}
\label{eqn_8.20} \frac{3C^2}{\nu}\int_{-t}^{0}\vert
z_\alpha(s)\vert_{\mathrm{X}}^2 \,ds  < \frac{\nu \lambda_1 t}{2},
\; t\geq t_0, \end{equation}
\end{corollary}

Finally we define a map $\varphi=\varphi_\alpha : \mathbb{R}_+ \times
\Omega \times \mathrm{H} \to \mathrm{H}$ by
\begin{equation}
\label{eqn:varphi} (t, \omega, x) \mapsto
v\big(t,z_\alpha(\omega)\big)\big(x-z(\omega)(0)\big)+z_\alpha(\omega)(t) \in
\mathrm{H},
\end{equation}
where $v(t,v_0)$, $t \geq 0$,   is  a  solution  to the following
problem
\begin{eqnarray}
\frac{dv}{dt}&=&-\nu \mathrm{A}v-B(v)-B(v,z)-B(z,v)-B(z)+\alpha z
+f, \label{eqn:4.10} \\ \hspace{0.5truecm} v(0)&=&v_0.
\label{eqn:4.10-a}
\end{eqnarray}
Because of Theorem 4.5 from \cite{Brz_Li_2006b}, which, for the completeness sake,  we state below as Theorem \ref{thm:4.2}, and because
$z_\alpha(\omega) \in C_{1/2}(\mathbb{R},\mathrm{X})$,
$z_\alpha(\omega)(0)$ is a well defined element of $\mathrm{H}$. Consequently,
the map $\varphi_\alpha$ is well defined.

\begin{definition}
Suppose that $z\in L^4_{\mathrm{loc}}\big([0,\infty); \mathbb{L}^4(\mathcal{O})\big)
\cap L^4_{\mathrm{loc}}([0,\infty);\mathrm{V})$, $f\in
\mathrm{V}^\prime$ and $v_0\in \mathrm{H}$. A function $v \in
C([0,\infty);\mathrm{H})\cap
L_{\mathrm{loc}}^2([0,\infty);\mathrm{\mathrm{V}^\prime})\cap
L^4_{\mathrm{loc}}\big([0,\infty); \mathbb{L}^4(\mathcal{O})\big)$ is a solution to
problem (\ref{eqn:4.10})-(\ref{eqn:4.10-a}) iff $v(0)=v_0$  and
(\ref{eqn:4.10}) holds in the weak sense, i.e. for any $\varphi\in
\mathrm{V}$
\begin{equation}\label{eqn:4.14}
\frac{d}{dt}(v(t),\varphi)=-\nu \big(\!\big(v(t),\varphi\big)\!\big)-b\big(v(t)
+z(t),\varphi,v(t) + z(t)\big)   +(\alpha z(t) +f,\varphi),
\end{equation}in the  distributions sense on  $(0,\infty)$. \end{definition}

\begin{theorem}\label{thm:4.2}
Assume that $\alpha \ge 0$,  $v_0\in
\mathrm{H}$, $f\in \mathrm{V}^\prime$ and\\ $z\in
L^4_{\mathrm{loc}}\big([0,\infty);\mathbb{L}^4(\mathcal{O})\big)\cap
L^2_{\mathrm{loc}}([0,\infty); \mathrm{V}^\prime)$.
\begin{enumerate}
\item[(i)]  Then there exists a unique solution $v$ of problem
(\ref{eqn:4.10})-(\ref{eqn:4.10-a}). \item[(ii)] If  in addition,
$v_0 \in \mathrm{V}$, $f\in \mathrm{H}$ and $z\in
C(\mathbb{R};V)\cap L^2_{\mathrm{loc}}\big(\mathbb{R};D(\mathrm{A})\big)$,
then $v \in C([0, \infty);\mathrm{V}) \cap
L_{\mathrm{loc}}^2\big([0,\infty);D(\mathrm{A})\big)$.
\end{enumerate}
\end{theorem}

It was proved in \cite[Proposition 6.16]{Brz_Li_2006b} that the map
$\varphi_\alpha$ does not depend on $\alpha$ and hence, from now on,  it will
be denoted  by $\varphi$. Furthermore, we have the
following result, see  \cite[Theorems 6.15 and 8.8]{Brz_Li_2006b}.

\begin{theorem}\label{thm_RDS}
Suppose that $\mathcal{O} \subset \mathbb{R}^2$ is a
Poincar{\'e} domain and  that
Assumption \ref{ass:A-01} is satisfied. Then the map $\varphi$   is an asymptotically compact  RDS
over   the metric DS
$\left(\hat{\Omega}(\xi,\mathrm{E}),\hat{\mathcal{F}},
\hat{\mathbb{P}},\hat{\vartheta}\right)$.
\end{theorem}

Our previous results yield the existence and the uniqueness of
solutions to problem \eqref{eqn_SNSE_add_noise}   
as well as its continuous dependence on the data (in particular on
the initial value $u_0$ and the force $f$). Moreover, if we
define, for $ x\in \mathrm{H}$, $\omega \in \Omega$, and $t \geq
s$,
\begin{equation}\label{eqn:4.12}
u(t,s;\omega,u_0):=\varphi(t-s;\vartheta_{s}\omega)u_0=v\big(t,s;\omega,u_0-z(s)\big)+z(t),
\end{equation}
then for each $s\in \mathbb{R}$ and each $u_0\in \mathrm{H}$, the
process $u(t)$, $t \geq s$, is a solution to problem \eqref{eqn_SNSE_add_noise}
.

Let us now recall Lemma 8.3 and Lemma 8.5 from \cite{Brz_Li_2006b} in which   $\lambda_1$ is the constant appearing in the Poincar{\'e} inequality \eqref{2.0} and
\begin{equation}\label{def-newnorm}
[v]^2:=\nu\|v\|^2-\displaystyle\frac{\lambda_1}{2}|v|^2,\; v\in V.
\end{equation}

\begin{lemma}\label{lem_8.3}
Suppose that $v$ is a solution to problem \eqref{eqn:4.10} on the
time interval $[a, \infty)$ with $z\in L^4_{\mathrm{loc}}\big([a,
\infty),\mathbb{L}^4(\mathcal{O})\big)\cap L^2_{\mathrm{loc}}([a,
\infty),\mathrm{V}^\prime)$ and   $\alpha \geq 0$.  Denote $\beta =\frac{\nu\lambda_1}2$ and
$g(t)=\alpha z(t) -B\big(z(t),z(t)\big)$, $t\in [a, \infty)$. Then,
 for any $t \ge\tau \geq a$
\begin{equation}\label{eqn:5.10b}\begin{aligned} \vert v(t)
\vert ^2 \leq &\vert v(\tau)\vert^2
e^{-\nu\lambda_1(t-\tau)+\frac{3C^2}{\nu}\int^t_{\tau}\vert
z(s)\vert^2_{\mathbb{L}^4}ds}\cr &+\frac{3}{\nu
}\int^t_{\tau}\{\vert g(s)\vert_{\mathrm{V}^\prime}^2+\vert f
\vert^2\}e^{-\nu\lambda_1(t-s)+\frac{3C^2}{\nu}\int_s^t(\vert
z(\zeta)\vert^2_{\mathbb{L}^4})d\zeta}ds.\end{aligned}
\end{equation}
\begin{equation}\label{eqn:5.3a}\begin{aligned}
\vert v(t)\vert ^2=&\vert v(\tau)\vert ^2e^{-\nu \lambda_1
(t-\tau)}+2\int_{\tau}^t e^{-\nu \lambda_1 (t-s)}(\langle
B(v(s),z(s)),v(t)\rangle\cr&+\langle g(s),v(s)\rangle+\langle f
,v(s)\rangle -[v(s)]^2)\,ds.
\end{aligned}
\end{equation}
\end{lemma}

\begin{lemma}\label{lem_8.5} Under the above assumptions, for each $\omega\in\Oe$,
\begin{equation*}
\lim_{t \to -\infty}\vert z(\omega)(t)\vert^2 e^{\nu \lambda_1t+
\int^{0}_{t}\frac{3C^2}{\nu} \vert
z(\omega)(s)\vert_{\mathbb{L}^4}^2 \,ds } =0.
\end{equation*}
\end{lemma}
Finally, let us  recall  a result containing in itself
\cite[Lemmas  8.6 and 8.7]{Brz_Li_2006b}.

\begin{lemma}\label{lem_8.6} Under the above assumptions, for each $\omega\in\Oe$,
\begin{equation*}
\int_{-\infty}^0\big[1+ \vert z(\omega)(t)\vert^2_{\mathbb{L}^4}
+\vert z(\omega)(t)\vert^4_{\mathbb{L}^4}\big] e^{\nu \lambda_1t+
\int^{0}_{t}\frac{3C^2}{\nu} \vert
z(\omega)(s)\vert_{\mathbb{L}^4}^2 \,ds }\,dt<\infty.
\end{equation*}
\end{lemma}
Since the proof of Lemma 8.6 from  \cite{Brz_Li_2006b} is miraculously
missing from the final version of that paper,  below we will present a detailed proof of Lemma \ref{lem_8.6}.
In fact, it is enough to consider the integral with the $4$th
moment of $z$ as the cases of $1$ and  of the $2$nd moment follow analogously.
\begin{proof} It is enough to consider the case of $\vert z(\omega)(t)\vert^4_{\mathbb{L}^4}$.
Let us fix $\omega \in \Omega$. By Corollary \ref{cor-ergodicestimates}  we
can find $t_0\geq 0$ such that for $t\geq t_0$,
$$
\int^{-t_0}_{-t}(-\nu \lambda_1+\frac{3C^2}{\nu}\vert
z(s)\vert_{\mathbb{L}^4}^2)\, ds\leq -\frac{\nu
\lambda_1(t-t_0)}{2}.
$$
By the continuity of all relevant functions, it is sufficient to prove
that the integral $\int_{-\infty}^{-t_0} \vert
z(\omega)(t)\vert^4_{\mathbb{L}^4} e^{\nu \lambda_1t+
\int^{0}_{t}\frac{3C^2}{\nu} \vert
z(\omega)(s)\vert_{\mathbb{L}^4}^2 \,ds }\,dt$ is finite.

Because of inequality  \eqref{eqn:A-sublinear}, we can find  a constant
$\rho_2=\rho_2(\omega)$, such that
\begin{equation*}\label{eqn:001}
\frac{\vert z(t) \vert_{\mathbb{L}^4}}{1+\vert t \vert} \leq \rho_2,
\; t\in\mathbb{R} .
\end{equation*}
Therefore, with $\rho_3(\omega):=
e^{\int^0_{-t_0}(-\nu\lambda_1+\frac{3C^2}{\nu}\vert
z(r)\vert_{\mathbb{L}^4}^2 )\,dr}<\infty$, we have, for every $\omega\in\Omega$,

\begin{eqnarray*}\label{eqn:c4}
&&\hspace{-2truecm}\lefteqn{\int_{-\infty}^{-t_0}\vert z(s)\vert_{\mathbb{L}^4} ^4
e^{\int^0_{s}(-\nu\lambda_1+\frac{3C^2}{\nu}\vert
z(r)\vert_{\mathbb{L}^4}^2 )\,dr}\,ds}   \\
&=& \rho_3
 \int_{-\infty}^{-t_0}\vert
z(s)\vert_{\mathbb{L}^4} ^4 e^{\int^{-t_0}_{s}(-\nu\lambda_1+\frac{3C^2}{\nu}\vert z(r)\vert_{\mathbb{L}^4}^2 )\,dr} \,ds \\
 &&\hspace{1truecm}\lefteqn{\leq   \rho_2^4 \rho_3 e^{\frac{\nu\lambda_1}{2}t_0} \int_{-\infty}^{-t_0} \vert s\vert^4 e^{\frac{\nu\lambda_1}{2}s}\, ds<\infty.}
\end{eqnarray*}
\end{proof}

\begin{definition}\label{def-R+D}
A function $r:\Omega\to
(0,\infty)$ belongs to the class $\mathfrak{R}$  if and only if
\begin{equation}
\label{eqn_R} \limsup_{t\to\infty}r(\vartheta_{-t}\omega)^2
e^{-\nu\lambda_1t
+\frac{3C^2}{\nu}\int_{-t}^0|z(\omega)(s)|^2_{\mathbb{L}^4}\, ds }=0,
\end{equation}where $C>0$ is the constant appearing in \eqref{universalC}.\\
We denote by $\mathfrak{DR}$ the class of all closed and bounded  random  sets
$D$ on $H$ such that
the function radious $\Omega\ni \omega \mapsto r(D(\omega)):=\sup\{|x|_H:x\in B\}$ belongs to the
 class $\mathfrak{R}$.\end{definition}

Observe that by Corollary \ref{cor-ergodicestimates}, the constant functions belong to
$\mathfrak{R}$. The following result lists a couple of other
important examples of functions belonging to class $\mathfrak{R}$.
\begin{proposition}
\label{prop_class_R} Define functions $r_i:\Omega \to (0,\infty)$,
$i=1,2,3$ by the following formulae, for $\omega\in \Omega$,
\begin{eqnarray*}
r_1^2(\omega)&:=& |z(\omega)(0)|_H^2, \\
r_2^2(\omega)&:=& \sup_{s\leq 0}\vert z (\omega)(s)\vert_H^2
e^{\nu\lambda_1s+\frac{3C^2}{\nu}\int^0_{s}\vert
z(\omega)(r)\vert^2_{\mathbb{L}^4}\,dr} \\
r_3^2(\omega)&:=& \int_{-\infty}^0 \vert z (\omega)(s)\vert^2_H
e^{\nu\lambda_1s+\frac{3C^2}{\nu}\int^0_{s}\vert
z(\omega)(r)\vert^2_{\mathbb{L}^4}\,dr}\,ds \\
r_4^2(\omega)&:=& \int_{-\infty}^0 \vert z
(\omega)(s)\vert^4_{\mathbb{L}^4}
e^{\nu\lambda_1s+\frac{3C^2}{\nu}\int^0_{s}\vert
z(\omega)(r)\vert^2_{\mathbb{L}^4}\,dr}\,ds\\
r_5^2(\omega)&:=& \int_{-\infty}^0
e^{\nu\lambda_1s+\frac{3C^2}{\nu}\int^0_{s}\vert
z(\omega)(r)\vert^2_{\mathbb{L}^4}\,dr}\,ds.
\end{eqnarray*}
Then all these functions belong to class $\mathfrak{R}$.\\
The class $\mathfrak{R}$ is closed  with respect to sum,
multiplication by a constant and if $r\in\mathfrak{R}$, $0\leq
\bar{r}\leq r$, then $\bar{r}\in\mathfrak{R}$.
\end{proposition}
\begin{proof} Since by Theorem \ref{thm:z-theta},  $z(\vartheta_{-t}\omega)(s)=z(\omega)(s-t)$, we have
\begin{eqnarray*}
r_2^2(\vartheta_{-t} \omega)&=& \sup_{s\leq 0}\vert z
(\vartheta_{-t} \omega)(s)\vert^2
e^{\nu\lambda_1s+\frac{3C^2}{\nu}\int^0_{s}\vert
z(\vartheta_{-t} \omega) (r)\vert^2_{\mathbb{L}^4}\,dr} \\
&=& \sup_{s\leq 0}\vert z (\omega)(s-t)\vert^2 e^{
\nu\lambda_1s+\frac{3C^2}{\nu}\int^0_{s}\vert
z(\omega)(r-t)\vert^2_{\mathbb{L}^4}\,dr} \\
&=& \sup_{s\leq 0}\vert z (\omega)(s-t)\vert^2
e^{\nu\lambda_1(s-t)+\frac{3C^2}{\nu}\int^{-t}_{s-t}\vert
z(\omega)(r)\vert^2_{\mathbb{L}^4}\,dr} e^{\nu\lambda_1t} \\
&=& \sup_{\sigma \leq -t}\vert z (\omega)(\sigma)\vert^2
e^{\nu\lambda_1\sigma+\frac{3C^2}{\nu}\int^{-t}_{\sigma}\vert
z(\omega)(r)\vert^2_{\mathbb{L}^4}\,dr} e^{\nu\lambda_1t} \\
\end{eqnarray*}
Hence, multiplying the above by $e^{-\nu\lambda_1t}
e^{\frac{3C^2}{\nu}\int_{-t}^0\vert
z(\omega)(r)\vert^2_{\mathbb{L}^4}\,dr}$ we get
\begin{equation*}
r_2^2(\vartheta_{-t} \omega)e^{-\nu\lambda_1t+
\frac{3C^2}{\nu}\int_{-t}^0\vert
z(\omega)(r)\vert^2_{\mathbb{L}^4}\,dr}\leq  \sup_{\sigma \leq
-t}\vert z (\omega)(\sigma)\vert^2
e^{\nu\lambda_1\sigma+\frac{3C^2}{\nu}\int^{0}_{\sigma}\vert
z(\omega)(r)\vert^2_{\mathbb{L}^4}\,dr}.
\end{equation*}
This,  together with  Lemma \ref{lem_8.5} concludes the proof in
the case of function $r_2$.  In the case of $r_1$, we have
\begin{equation*}
r_1^2(\vartheta_{-t} \omega)e^{-\nu\lambda_1t+
\frac{3C^2}{\nu}\int_{-t}^0\vert
z(\omega)(r)\vert^2_{\mathbb{L}^4}\,dr}=  \vert z
(\omega)(-t)\vert^2 e^{-\nu\lambda_1
t+\frac{3C^2}{\nu}\int^{0}_{-t}\vert
z(\omega)(r)\vert^2_{\mathbb{L}^4}\,dr}.
\end{equation*}
Thus, by Lemma \ref{lem_8.5} we infer that $r_1$ also belongs to
the class $\mathfrak{R}$. The argument in the case of function
$r_3$ is similar but  the completness sake we include it here.
From the first part of the proof we infer that
\begin{equation*}
r_3^2(\vartheta_{-t} \omega)e^{-\nu\lambda_1t+
\frac{3C^2}{\nu}\int_{-t}^0\vert
z(\omega)(r)\vert^2_{\mathbb{L}^4}\,dr}\leq
\int_{-\infty}^{-t}\vert z (\omega)(\sigma)\vert^2
e^{\nu\lambda_1\sigma+\frac{3C^2}{\nu}\int^{0}_{\sigma}\vert
z(\omega)(r)\vert^2_{\mathbb{L}^4}\,dr}\, d\sigma.
\end{equation*}
Since  by Lemma \ref{lem_8.6} $\int_{-\infty}^{0}\vert z
(\omega)(\sigma)\vert^2
e^{\nu\lambda_1\sigma+\frac{3C^2}{\nu}\int^{0}_{\sigma}\vert
z(\omega)(r)\vert^2_{\mathbb{L}^4}\,dr}\, d\sigma$ is finite, by
the Lebesgue monotone Theorem we conclude that
$$\int_{-\infty}^{-t}\vert z (\omega)(\sigma)\vert^2
e^{\nu\lambda_1\sigma+\frac{3C^2}{\nu}\int^{0}_{\sigma}\vert
z(\omega)(r)\vert^2_{\mathbb{L}^4}\,dr}\, d\sigma \to 0 \mbox{ as } t\to\infty.$$ The proof in the other cases is analogous.
The proof of the second part of the Proposition is trivial. This
concludes the proof.  \end{proof}

Now we are ready to state and prove the main result of this
paper.
\begin{theorem}\label{thm:5.1}
Suppose that the domain $\mathcal{O} \subset \mathbb{R}^2$ is a
Poincar{\'e} domain and  that the
 Assumption \ref{ass:A-01} is satisfied.  Consider  the metric DS
$\mathfrak{T}=\left(\hat{\Omega}(\xi,\mathrm{E}),\hat{\mathcal{F}},
\hat{\mathbb{P}},\hat{\vartheta}\right)$ from Proposition
\ref{prop:metricDS}, and the RDS  $\varphi$   on $H$ over
$\mathfrak{T}$ generated by  the 2D stochastic Navier-Stokes equations with additive noise \eqref{eqn_SNSE_add_noise} satisfying Assumption A1.
Then
the following properties hold.
 \begin{trivlist}
 \item[(i)] there exists a $\mathfrak{DR}$-absorbing set $B\in\mathfrak{DR}$;
 \item[(ii)] the RDS $\varphi$ is $\mathfrak{DR}$-asymptotically compact;
\item[(iii)] the family  $A$ of sets defined by $A(\omega)=\Omega_B(\omega)$ for all $\omega\in\Omega,$ is the minimal $\mathfrak{DR}$-attractor for $\varphi,$
is $\hat{\mathcal{F}}$-measurable, and
\begin{equation}\label{RACDFH}
A(\omega)=\overline{\bigcup_{C\subset H}\Omega_C(\omega)}\quad\hat{\mathbb{P}}-a.s.,
\end{equation}where the union in \eqref{RACDFH}
is made for all bounded and closed nonempty deterministic subsets $C$
of $H$.
 \end{trivlist}
 \end{theorem}
 \begin{proof}
 In view of Theorem \ref{Teorema1} and Remark \ref{RACDFremark},
 it is enough to show points (i)-(ii). We prove now point (i). The proof of point (ii) will be done in the next section.

Let $D $ be a random  set from the class $\mathfrak{DR}$.  Let  $r_D(\omega)$ be
the  radius of $D(\omega)$, i.e. $r_D(\omega):=\sup\{|x|_H:x\in
D(\omega)\}$, $\omega\in\Omega$.

Let $\omega \in \Omega$ be fixed. For given  $s \leq 0$ and $x \in
\mathrm{H}$, let $v$ be the solution of (\ref{eqn:4.10}) on time
interval $[s, \infty)$ with the initial condition $v(s)= x-z(s)$.
Applying  (\ref{eqn:5.10b}) with   $t=0, \tau =s \leq 0 $,  we get

\begin{eqnarray}\label{eqn:5.13}
\vert v(0)\vert ^2 &\leq&  2\vert x\vert ^2
e^{\nu\lambda_1s+\frac{3C^2}{\nu}\int_{s}^0\vert
z(r)\vert^2_{\mathbb{L}^4}\, \,dr}+2\vert z(s)\vert ^2
e^{\nu\lambda_1s+\frac{3C^2}{\nu}\int_{s}^0\vert
z(r)\vert^2_{\mathbb{L}^4}\, \,dr} \nonumber\\ &
+&\frac{3}{\nu}\int_{s}^0\{\Vert
g(t)\Vert^2_{\mathrm{V}^\prime}+\Vert f \Vert ^2_{\mathrm{V}^\prime}
\}e^{\nu\lambda_1t+\frac{3C^2}{\nu}\int_{t}^0\vert
z(r)\vert^2_{\mathbb{L}^4}\, dr}\,dt.
\end{eqnarray}
Set, for $\omega\in\Omega$,%
\begin{eqnarray}
 r_{11}(\omega)^2= 2 +\sup_{s\leq 0}\left\{2\vert
z(s)\vert ^2 e^{\nu\lambda_1s+\frac{3C^2}{\nu}\int_{s}^0\vert
z(r)\vert^2_{\mathbb{L}^4}\, \,dr} \right.
\nonumber\\
+\left.\frac{3}{\nu}\int_{s}^0\{\Vert
g(t)\Vert^2_{\mathrm{V}^\prime}+\Vert f \Vert ^2_{\mathrm{V}^\prime}
\}e^{\nu\lambda_1t+\frac{3C^2}{\nu}\int_{t}^0\vert
z(r)\vert^2_{\mathbb{L}^4}\, dr}\,dt \right\},
 \label{eqn:r_11}\\
\label{eqn:r_12} r_{12}(\omega)=\vert
z(0)(\omega)\vert_{\mathrm{H}}.
\end{eqnarray}

By Lemma \ref{lem_8.6} and Proposition \ref{prop_class_R}  we infer
that both $r_{11}$ and $r_{12}$ belong to $\mathfrak{R}$ and also
that $r_{13}:=r_{11}+r_{12}$  belongs to $\mathfrak{R}$ as well.
Therefore  the random set $B$ defined by   $B(\omega):=\{u \in
\mathrm{H}: \vert u\vert  \leq r_{13}(\omega)\}$ belongs to the
family $\mathfrak{DR}$.

We will show now that $B$ absorbs $D$. Let $\omega\in\Omega$ be
fixed. Since $r_D\in\mathfrak{R}$ there exists $t_D(\omega)\geq
0$, such that
$$r_0(\vartheta_{-t}\omega)^2 e^{-\nu\lambda_1t
+\frac{3C^2}{\nu}\int_{-t}^0|z(\omega)(s)|^2_{L^4}\, ds }\leq 1,
\mbox{ for } t\geq t_D(\omega).$$ Thus, if $x\in
D(\vartheta_{-t}\omega)$ and $s\geq t_D(\omega)$, then by
\eqref{eqn:5.13},  $\vert v(0,\omega;s,x-z(s))\vert \leq
r_{11}(\omega)$. Thus we infer that \begin{equation*} \vert
u(0,s;\omega,x)\vert  \leq \vert v(0,s;\omega,x-z(s))\vert + \vert
z(0)(\omega)\vert \leq r_{13}(\omega).
\end{equation*}
In other words, $u(0,s;\omega,x)\in B(\omega)$, for all $s\geq
t_D(\omega)$. This proves that $B$ absorbs $D$.
 \end{proof}

\section{Proof of the $\mathfrak{DR}$-asymptotical compactness property
of the  RDS $\varphi$ generated by stochastic NSEs} \label{sec:as_com}

We consider here  the RDS  $\varphi$ over the metric DS
$\left(\hat{\Omega}(\xi,\mathrm{E}),\hat{\mathcal{F}},\hat{\mathbb{P}},\hat{\vartheta}\right)$, from Proposition \ref{prop:metricDS} and the family $\mathfrak{DR}$ defined in Definition \ref{def-R+D}.
The main result in  this
section is the following result.

\begin{proposition}\label{prp:5.1}  Assume that for each random set $D$   belonging to $\mathfrak{DR}$,    there
exists a   random set $B$ belonging to $\mathfrak{DR}$ such
that $B$ absorbs $D$. Then, the RDS $\varphi$ is
$\mathfrak{DR}$-asymptotically compact.
\end{proposition}

Let us recall that the RDS $\varphi$ is independent of the
auxiliary parameter $\alpha \in \mathbb{N}$.  For reasons that
will become clear in the course of the proof we choose $\alpha$
such that $\mathbb{E}\vert z_\alpha(0) \vert_{\mathbb{L}^4} ^2
\leq \frac{\nu^2\lambda_1}{6C^2}$, where $z_{\alpha}(t)$, $t\in
\mathbb{R}$ is the Ornstein-Uhlenbeck process from section
\ref{sec-appl-sNSEs}, $C>0$ is a certain universal constant. Such
a choice is possible because of Proposition \ref{prop:metricDS}.
Let us choose $\alpha\in \mathbb{N}$ such that the condition
(\ref{universalC}) is satisfied.

For simplicity of notation we will denote the space
$\hat{\Omega}(\xi,\mathrm{E})$ simply by $\Omega$ and the process
$z_{\alpha}(t)$, $t\in \mathbb{R}$ by $z(t)$, $t\in \mathbb{R}$.

\begin{proof} Suppose that  $D  $ is a closed random  set from the class $\mathfrak{DR}$ and
$B\in\mathfrak{DR}$ is a closed random set which absorbs $D$. We fix $\omega\in\Omega$.  Let us take  an increasing sequence of positive numbers
$(t_n)_{n=1}^\infty$   such that $t_n \to \infty$ and an
   $\mathrm{H}$-valued sequence $(x_n)_n$ such that $x_n \in D(\vartheta_{-t_n}\omega)$, for all $n\in\mathbb{N}$.

\noindent \textbf{Step I.} Reduction. Since $B$ absorbs $D$,
$\varphi(t_n,\vartheta_{-t_n}\omega,D(\vartheta_{-t_n}\omega))
\subset B(\omega)$ for $n\in \mathbb{N}$ sufficiently large.
Since $B(\omega)$ as a  bounded set in $H$ is   weakly
pre-compact in $H$,  without loss of generality we may assume that
$$\varphi(t_n,\vartheta_{-t_n}\omega),D(\vartheta_{-t_n}\omega))\subset B(\omega)$$ for all
$n\in \mathbb{N}$  and,  for some $y_0 \in H$,
\begin{equation}\label{eqn:5.23}
\varphi(t_{n},\vartheta_{-t_{n}}\omega,x_{n})\rightharpoonup  y_0
\quad \hbox{weakly in }\mathrm{H}.
\end{equation}
Our aim  is to prove that for some subsequence
\begin{equation}\label{strong}\varphi(t_{n^\prime},\vartheta_{-t_{n^\prime}}\omega, x_{n^\prime})
\rightarrow y_0 \hbox{ strongly in }\mathrm{H}.
\end{equation}

Since $z(0) \in \mathrm{H}$, then
\begin{equation}\label{eqn:5.23-a}
\varphi(t_{n},\vartheta_{-t_{n}}\omega, x_{n}-z(0)) \rightarrow
y_0-z(0) \quad \hbox{weakly in }\mathrm{H}.
\end{equation}
In particular,
\begin{equation}\label{eqn:5.23-b}
\vert y_0-z(0)\vert \leq \liminf_{n\rightarrow\infty}
\vert\varphi(t_{n},\vartheta_{-t_{n}}\omega, x_{n}-z(0))\vert.
\end{equation}

Arguing as in \cite{Brz_Li_2006b},  we can show that in order to
prove \eqref{strong} it is enough to prove that for some
subsequence $\{{n^\prime}\}\subset \mathbb{N}$
\begin{equation}\label{eqn:5.24}
\vert y_0-z(0)\vert \geq \limsup_{n^\prime\rightarrow \infty}\vert
\varphi(t_{n^\prime},\vartheta_{-t_{n^\prime}}\omega,x_{n^\prime})-z(0)\vert.
\end{equation}

 \noindent \textbf{Step II.} Construction of a negative trajectory, i.e. a
 sequence $(y_n)_{n=-\infty}^0$ such that $y_n \in
B(\theta_n\omega)$, $n \in \mathbb{Z}^-$, and
$$y_k=\varphi(k-n,\theta_n\omega,y_n), \;n<k\le 0.$$

Since $B$ absorbs $D$, there exists a constant $N_1(\omega)\in
\mathbb{N}$, such that
$$\{\varphi(-1+t_{n},
\vartheta_{1-t_{n}}\vartheta_{-1}\omega,x_{n}):n \geq N_1(\omega)\}
\subset B(\vartheta_{-1}\omega).$$ Hence we can find  a
subsequence $\{n^{\prime}\} \subset \mathbb{N}$ and $y_{-1} \in
B(\vartheta_{-1}\omega)$ such that
\begin{equation}\label{eqn:6.400}
\varphi(-1+t_{n^\prime},\vartheta_{-t_{n^\prime}}\omega, x_{n^\prime})
\rightharpoonup y_1 \hbox{ weakly in }\mathrm{H}.
\end{equation}
Let us observe that the cocycle property, with $t=1,
s=t_{n^\prime}-1$, and $\omega$ being replaced by $ \vartheta
_{-t_{n^\prime}}\omega$, reads as follows: $$\varphi(t_{n^\prime},
\vartheta_{-t_{n^\prime}}\omega)=\varphi(1, \vartheta_{-1}\omega)
\varphi(-1+t_{n^\prime}, \vartheta_{-t_{n^\prime}}\omega).$$

Hence, by Lemma 
\cite[Lemma 7.2]{Brz_Li_2006b}, from  (\ref{eqn:5.23}) and
\eqref{eqn:6.400} we infer that $\varphi(1,
\vartheta_{-1}\omega,y_1)=y_0$. By induction, for each
$k=1,2,\ldots$, we can construct  a subsequence
 $\{n^{(k)}\}\subset \{n^{(k-1)}\}$ and  $y_{-k}
\in B(\vartheta_{-k}\omega)$,
 such that $\varphi(1, \vartheta_{-k}\omega,y_{-k})=y_{-k+1}$ and
\begin{equation}\label{eqn:5.1a}
\varphi(-k+t_{n^{(k)}},\vartheta_{-t_{n^{(k)}}}\omega,x_{n^{(k)}})
\rightharpoonup y_{-k}  \quad \hbox{weakly in }\mathrm{H}, \quad
\hbox{as } n^{(k)}\rightarrow \infty.
\end{equation}

 As above,  the  cocycle
property with $t=k$, $s=t_{n^{(k)}}$ and  $ \omega$ being replaced
by $\vartheta_{-t_{n^{(k)}}}\omega$, yields
\begin{equation}\label{eqn:5.4a}
\varphi(t_{n^{(k)}},\vartheta_{-t_{n^{(k)}}}\omega)=\varphi(k,
\vartheta_{-k}\omega)\varphi(t_{n^{(k)}}-k,\vartheta_{-t_{n^{(k)}}}\omega),
\;
 k\in \mathbb{N}.\end{equation}
 Hence, from
(\ref{eqn:5.1a}) by applying   \cite[Lemma 7.1]{Brz_Li_2006b}, we
get
\begin{eqnarray}\label{eqn:5.2a}
 y_0&=&\hbox{w-}\lim_{n^{(k)}\rightarrow\infty}
 \varphi(t_{n^{(k)}},\vartheta_{-t_{n^{(k)}}}\omega, x_{n^{(k)}})\\
 &=&\hbox{w-}\lim_{n^{(k)}\rightarrow\infty}
 \varphi(k,\vartheta_{-k}\omega,\varphi(t_{n^{(k)}}-k,\vartheta_{-t_{n^{(k)}}}\omega,x_{n^{(k)}}))
 \nonumber\\
&=&\varphi(k,\vartheta_{-k}\omega,(\hbox{w-}\lim_{n^{(k)}\rightarrow\infty}
\varphi(t_{n^{(k)}}-k,\vartheta_{-t_{n^{(k)}}}\omega)x_{n^{(k)}}))
=\varphi(k,\vartheta_{-k}\omega,y_{-k}),
\nonumber
\end{eqnarray}
where $\hbox{w-}\lim$ denotes the limit in  the weak topology on
$\mathrm{H}$. The same proof yields a more general property:
\begin{equation*}\varphi(j,
\vartheta_{-k}\omega,y_{-k})=y_{-k+j},\quad \mbox {if } 0\leq j \leq k.
\end{equation*}

Before we continue our proof, let us point out that,
\eqref{eqn:5.2a} means precisely that $y_0=u(0,-k;\omega,y_{-k})$,
where $u$ is defined by \eqref{eqn:4.12}.

\noindent \textbf{Step III.} Proof of (\ref{eqn:5.24}).
 From now on, until explicitly stated, we fix
$k\in \mathbb{N}$, and  we will consider problem \eqref{eqn_SNSE_add_noise} 
on the time interval $[-k,0]$. From (\ref{eqn:4.12}) and
(\ref{eqn:5.4a}), with $t=0$ and $ s=-k$, we have
\begin{eqnarray}\label{eqn:*1}
&&\hspace{-2truecm}\lefteqn{\vert \varphi(t_{n^{(k)}},\vartheta_{-t_{n^{(k)}}}\omega,
x_{n^{(k)}} -z(0))\vert^2 }\\ 
&=&\vert \varphi(k,
\vartheta_{-k}\omega,\varphi(t_{n^{(k)}}-k,\vartheta_{-t_{n^{(k)}}}\omega,
x_{n^{(k)}}-z(0)))\vert^2 \nonumber\\
&=&\vert v(0,\omega;-k,
\varphi(t_{n^{(k)}}-k,\vartheta_{-t_{n^{(k)}}}\omega,
x_{n^{(k)}})-z(-k))\vert^2.
\nonumber
\end{eqnarray}

Let $v$ be the  solution to (\ref{eqn:4.10}) on $[-k, \infty)$
with $z=z_\alpha(\cdot, \omega)$ and the initial condition at time
$-k$: $v(-k)=\varphi(t_{n^{(k)}}-k,\vartheta_{-t_{n^{(k)}}}\omega,
x_{n^{(k)}})-z(-k)$. In other words,
$$v(s)=v\Bigl(s,-k;\omega,\varphi(t_{n^{(k)}}-k,\vartheta_{-t_{n^{(k)}}}\omega,
x_{n^{(k)}})-z(-k)\Bigr), \quad s\geq-k.$$ From
(\ref{eqn:5.3a}) with $t=0$ and $ \tau=-k$ we infer that

\begin{eqnarray}\label{eqn:4.9c}
&&\hspace{-1truecm}\lefteqn{\vert\varphi(t_{n^{(k)}},\vartheta_{-t_{n^{(k)}}}\omega,
x_{n^{(k)}}) -z(0)\vert^2}
\\ &=&e^{-\nu \lambda_1 k}\vert
\varphi(t_{n^{(k)}}-k,\vartheta_{-t_{n^{(k)}}}\omega,x_{n^{(k)}})-z(-k)\vert
^2
\nonumber\\
&+&\, 2\int_{-k}^0 e^{\nu \lambda_1 s}\left( b(v(s),z(s)),v(s))\right.
+\langle g(s),v(s)\rangle +\left. \langle f
,v(s)\rangle-[v(s)]^2\right)\,ds.
\nonumber
\end{eqnarray}

To finish the proof it is enough to find a non-negative function $h \in
L^1(-\infty,0)$ such that
\begin{equation}\label{eqn:5.15b}
\limsup_{n^{(k)}\rightarrow\infty}\vert
\varphi(t_{n^{(k)}},\vartheta_{-t_{n^{(k)}}}\omega,x_{n^{(k)}})-z(0)\vert
^2\leq \int_{-\infty}^{-k}\, h(s)\, ds +\vert y_0-z(0)\vert ^2.
\end{equation}
For, if we define the diagonal process
$(m_j)_{j=1}^\infty$ by $m_j=j^{(j)}$, $j \in \mathbb{N}$, then for
each $k\in \mathbb{N}$, the sequence $(m_j)_{j=k}^\infty$ is a
subsequence of the sequence $(n^{(k)})$ and hence by
\eqref{eqn:5.15b}, $\limsup_{j}\vert
\varphi(t_{m_j},\vartheta_{-t_{m_j}}\omega,x_{m_j})-z(0)\vert ^2\leq
\int_{-\infty}^{-k}\, h(s)\, ds+\vert y_0-z(0)\vert ^2$. Taking the
$k\to \infty$ limit in the last inequality we infer that
$\limsup_{j}\vert
\varphi(t_{m_j},\vartheta_{-t_{m_j}}\omega,x_{m_j})-z(0)\vert ^2\leq
\vert y_0-z(0)\vert ^2$ which proves claim \eqref{eqn:5.24}.

\noindent \textbf{Step IV.} Proof of (\ref{eqn:5.15b}). We begin
with estimating  the first term on the RHS of \eqref{eqn:4.9c}. If
$-t_{n^{(k)}}<-k $, then by \eqref{eqn:4.12} and \eqref{eqn:5.10b}
we infer that
\begin{eqnarray}\label{eqn:5.10d}
&&\hspace{-1 truecm}\lefteqn{
\vert\varphi(t_{n^{(k)}}-k,\vartheta_{-t_{n^{(k)}}}\omega,x_{n^{(k)}})-z(-k)\vert^2e^{-\nu
\lambda_1 k}} \nonumber \\
&&\hspace{3 truecm}\lefteqn{ =\, \vert v(-k,
-t_{n^{(k)}};\vartheta_{-k}\omega,x_{n^{(k)}}-z(-t_{n^{(k)}}))\vert
^2e^{-\nu \lambda_1 k}} \nonumber\\
&\leq &e^{-\nu \lambda_1 k}\Bigl\{\vert
x_{n^{(k)}}-z(-t_{n^{(k)}})\vert^2
e^{-\nu\lambda_1(t_{n^{(k)}}-k)+\frac{3C^2}{\nu}\int^{-k}_{-t_{n^{(k)}}}\vert
z(s)\vert^2_{\mathbb{L}^4}\, ds} \nonumber\\
 &\lefteqn{ \frac{3}{\nu
}\int^{-k}_{-t_{n^{(k)}}}\Bigl[\Vert
g(s)\Vert_{\mathrm{V}^\prime}^2+\Vert f
\Vert_{\mathrm{V}^\prime}^2\Bigr]
e^{-\nu\lambda_1(-k-s)+\frac{3C^2}{\nu}\int_s^{-k}(\vert z(\zeta
)\vert^2_{\mathbb{L}^4})d\zeta }ds\Bigr\}} \\ &\leq&
2I_{n^{(k)}}+2I\!I_{n^{(k)}}+\frac{3}{\nu}I\!I\!I_{n^{(k)}}+\frac{3}{\nu
}\Vert f \Vert_{\mathrm{V}^\prime}^2\, I\!V_{n^{(k)}}, \nonumber
\end{eqnarray}

where
\begin{eqnarray*}
I_{n^{(k)}} &=&\vert x_{n^{(k)}}\vert^2e^{-\nu\lambda_1t_{n^{(k)}}
+\frac{3C^2}{\nu}\int^{-k}_{-t_{n^{(k)}}}\vert z(s)\vert^2_{\mathbb{L}^4}\,ds},\\
I\!I_{n^{(k)}}&=&\vert
z(-t_{n^{(k)}})\vert^2e^{-\nu\lambda_1t_{n^{(k)}}
+\frac{3C^2}{\nu}\int^{-k}_{-t_{n^{(k)}}}\vert z(s)\vert^2_{\mathbb{L}^4}\,ds},\\
I\!I\!I_{n^{(k)}}&=&\int^{-k}_{-\infty}\Vert
g(s)\Vert_{\mathrm{V}^\prime}^2e^{\nu\lambda_1s+\frac{3C^2}{\nu}\int_s^{-k}\vert
z(\zeta )\vert^2_{\mathbb{L}^4}\,d\zeta }\,ds, \\
I\!V_{n^{(k)}} &=&
\int^{-k}_{-\infty}e^{\nu\lambda_1s+\frac{3C^2}{\nu}\int_s^{-k}\vert
z(\zeta )\vert^2_{\mathbb{L}^4}\, d\zeta }\,ds.
\end{eqnarray*}

First, we will  find a non-negative function $h \in
L^1(-\infty,0)$ such that

\begin{equation}\label{eqn:c6}\begin{aligned}
\hbox{\hspace{-0.8truecm}} &\limsup_{n^{(k)}\rightarrow \infty}
\vert
\varphi(t_{n^{(k)}}-k,\vartheta_{-t_{n^{(k)}}}\omega,x_{n^{(k)}})-z(-k)\vert
^2e^{-\nu \lambda_1 k} \cr & \hbox{\hspace{5.8truecm}}
  \leq \int_{-\infty}^{-k} h(s)\, ds, \quad k \in \mathbb{N}.
\end{aligned}
\end{equation}

For this we will need one more auxiliary result.

\begin{lemma}\label{lem_8.4}
$\limsup\limits_{n^{(k)}\to \infty}I_{n^{(k)}}=0$.
\end{lemma}
\begin{proof}
Let us recall that, $\alpha \in \mathbb{N}$, $z(t)=z_{\alpha}(t)$,
$t\in \mathbb{R}$, is  the Ornstein-Uhlenbeck process from section
\ref{sec-appl-sNSEs}, and $\mathbb{E}\vert z(0)
\vert_{\mathrm{X}}^2=\mathbb{E}\vert z_\alpha(0)
\vert_{\mathrm{X}}^2 < \frac{\nu^2\lambda_1}{6C^2}$. Let us also
recall that the space $\hat{\Omega}(\xi,\mathrm{E})$ was
constructed in such a way that
$$
\lim_{{n^{(k)}}\rightarrow
\infty}\frac{1}{-k-(-t_{n^{(k)}})}\int_{-t_{n^{(k)}}}^{-k}\vert
z_\alpha(s)\vert_{\mathrm{X}}^2 ds =\mathbb{E}\vert z(0)
\vert_{\mathrm{X}} ^2<\infty.
$$
Therefore, since the embedding $\mathrm{X} \embed
\mathbb{L}^4(\mathcal{O})$ is a contraction, we have for $n^{(k)}$
sufficiently large,
\begin{equation}\label{eqn:5.13c}
\frac{3C^2}{\nu}\int^{-k}_{-t_{n^{(k)}}}\vert
z(s)\vert_{\mathbb{L}^4}^2\, ds < \frac{\nu
\lambda_1}{2}(t_{n^{(k)}}-k).
\end{equation}
Since the set $D(\omega)$ is bounded in $\mathrm{H}$, there exists $\rho_1
>0$ such that
 $\vert x_{n^{(k)}}\vert \leq \rho_1$ for every $n^{(k)}$. Hence,
\begin{equation}\label{eqn:c1}\begin{aligned} &\limsup_{n^{(k)}\rightarrow\infty}\vert
x_{n^{(k)}}\vert^2e^{-\nu\lambda_1t_{n^{(k)}}
+\frac{3C^2}{\nu}\int^{-k}_{-t_{n^{(k)}}}\vert
z(s)\vert^2_{\mathbb{L}^4}ds} \\& \hspace{2.0 truecm}\leq
\limsup_{n^{(k)}\rightarrow\infty} \rho^2_1 e^{-\frac{\nu
\lambda_1}{2}(t_{n^{(k)}}-k)}=0.\quad \quad \quad \quad \quad
\quad \quad \quad \quad \quad  \end{aligned} \end{equation}
  \end{proof}

Therefore, by \eqref{eqn:5.10d}, and lemmas \ref{lem_8.5}, \ref{lem_8.6} and \ref{lem_8.4}, the proof of \eqref{eqn:c6} is concluded, and it only
remains to finish the proof of  inequality \eqref{eqn:5.15b}, which we are
going to do right now.

\textbf{The end of the proof of  inequality \eqref{eqn:5.15b}.}

Let us denote $\tilde{y}_k=y_k-z(-k)$ and
\begin{eqnarray*}
v^{n^{(k)}}(s)&=&v(s,-k;\omega,\varphi(t_{n^{(k)}}-k,\vartheta_{-t_{n^{(k)}}}\omega)
x_{n^{(k)}}-z(-k)),\; s \in (-k,0),\\
v_k(s)&=&v(s,-k;\omega,y_{-k}-z(-k)),\; s \in (-k,0).
\end{eqnarray*}

From property (\ref{eqn:5.1a}) and
   \cite[Lemma 7.1]{Brz_Li_2006b} we infer that
\begin{equation}\label{eqn:5.9b0}
v^{n^{(k)}}(\cdot)\to v_k \text{ weakly in } L^2(-k,0;
\mathrm{V}).
\end{equation}

Since $e^{\nu \lambda_1 \cdot}g(\cdot), e^{\nu \lambda_1
\cdot}f\in L^2(-k,0;\mathrm{V}^\prime)$, we get
\begin{equation}\label{eqn:004a}
\lim_{n^{(k)}\rightarrow \infty}\int_{-k}^0e^{\nu \lambda_1
s}\langle g(s),v^{n^{(k)}}(s)\rangle\,ds = \int_{-k}^0 e^{\nu
\lambda_1 s}\langle g(s), v_k(s) \rangle\,ds,
\end{equation}
and
\begin{equation}\label{eqn:004}
\begin{aligned}
&\lim_{n^{(k)}\rightarrow \infty}\int_{-k}^0 e^{\nu \lambda_1
s}\langle f,v^{n^{(k)}}(s)\rangle\,ds =\int_{-k}^0 e^{\nu
\lambda_1 s}\langle f, v_k(s)
 \rangle\,ds.
\end{aligned}
\end{equation}

On the other hand, using the same methods as those in the proof of
Theorem \ref{thm:4.2}, there exists a subsequence of
$\{v^{n^{(k)}}\}$,  which,  for the sake of simplicity of
notation, is denoted as the old one and which satisfies
\begin{equation}\label{eqn:5.9d}
\begin{aligned}
v^{n^{(k)}} \to v_k \quad \mbox{strongly in }
L^2(-k,0;\mathbb{L}_{loc}^2(D)).
\end{aligned}\end{equation}

 Next,
 since   $e^{\nu \lambda_1 t}z(t)$, $t\in \mathbb{R}$,  is an $\mathbb{L}^4$-valued process,
 Thus by \cite[Corollary 5.3]{Brz_Li_2006b}, 
\eqref{eqn:5.9b0} and \eqref{eqn:5.9d}, we infer that
\begin{equation}\label{eqn:003}
\begin{aligned} \lim_{n^{(k)}\to \infty} &\int_{-k}^0
e^{\nu \lambda_1 s} b(v^{n^{(k)}}(s),z(s)),v^{n^{(k)}}(s))  \,ds
\cr =& \int_{-k}^0e^{\nu \lambda_1 s} b(v_k(s), z(s),v_k(s)) \,
ds.
\end{aligned}
\end{equation}

 Moreover, since the norms $[\cdot]$ and  $\Vert \cdot \Vert$  are equivalent  on $\mathrm{V}$,
 and since for any $s\in (-k,0] $, $ e^{-\nu k \lambda_1}
 \le e^{\nu \lambda_1 s} \leq 1$,  $(\int_{-k}^0 e^{\nu \lambda_1
s}[\cdot]^2\,ds)^{1/2}$ is a norm in $L^2(-k, 0;\mathrm{V})$
equivalent to the standard one. Hence,  from \eqref{eqn:5.9b0} we
obtain,
$$\begin{aligned}
&\int_{-k}^0 e^{\nu \lambda_1s}[ v_k(s) ]^2\,ds \leq
\liminf_{n^{(k)}\rightarrow\infty}\int_{-k}^0 e^{\nu \lambda_1
s}[v^{n^{(k)}}(s)]^2 \,ds.
\end{aligned}
$$
In other words,
 \begin{equation}\label{eqn:005}
\limsup_{n^{(k)}\rightarrow\infty}\Bigl\{-\int_{-k}^0[v^{n^{(k)}}(s)]^2\,ds\Bigr\}
\leq -\int_{-k}^0 e^{\nu \lambda_1 s}[v_k(s) ]^2\,ds.
\end{equation}

From \eqref{eqn:4.9c},  eqref{eqn:c6} and \eqref{eqn:003}, and inequality
\eqref{eqn:005} we conclude that

\begin{equation}\label{eqn:5.15a}\begin{aligned}
&\limsup_{n^{(k)}\to \infty} \vert
\varphi(t_{n^{(k)}},\vartheta_{-t_{n^{(k)}}}\omega)x_{n^{(k)}}-z(0)\vert^2
 \\&\hspace{2.6 truecm} \leq \int_{-\infty}^{-k}\, h(s)\, ds  +  2
\int_{-k}^0 e^{\nu \lambda_1 s} \Bigl\{\langle B(v_k(s),z(s)), v_k(s) \rangle \\
&\hspace{4.6 truecm}+\langle g(s), v_k(s) \rangle+ \langle f,
v_k(s)
 \rangle -[v_k(s) ]^2\Bigr\}\,ds.
 \end{aligned}\end{equation}

On the other hand, from (\ref{eqn:5.2a}) and (\ref{eqn:5.3a}),
we have
\begin{eqnarray}\label{eqn:006}
\vert y_0-z(0)\vert ^2  &= & \vert \varphi(k,\vartheta_{-k}\omega
)y_k-z(0)\vert ^2 =\vert v(0,-k;\omega,y_k-z(-k))\vert ^2
\nonumber\\
&= &\vert y_k-z(-k)\vert ^2e^{-\nu \lambda_1 k} +2\int_{-k}^0
e^{\nu \lambda_1 s}\Bigl\{\langle g(s),v_k(s) \rangle
\\
&+&\langle B(v_k(s),z(s)),v_k(s) \rangle + \langle f,v_k(s)
\rangle -[v_k(s)]^2\,\Bigr\}\, ds. \nonumber
\end{eqnarray}

Hence, by  combining (\ref{eqn:5.15a}) with \eqref{eqn:006}, we
get
\begin{equation*}
\begin{aligned}
&\limsup_{n^{(k)}\rightarrow\infty}\vert
\varphi(t_{n^{(k)}},\vartheta_{-t_{n^{(k)}}}\omega)x_{n^{(k)}}-z(0)\vert
^2\cr \leq & \int_{-\infty}^{-k}\, h(s)\, ds +\vert y_0-z(0)\vert
^2 -\vert y_k-z(-k)\vert
^2e^{-\nu \lambda_1 k}\cr %
\leq& \int_{-\infty}^{-k}\, h(s)\, ds +\vert y_0-z(0)\vert ^2.
\end{aligned}
\end{equation*}
which proves \eqref{eqn:5.15b} and hence
 the proof of Proposition \ref{prp:5.1} is finished.
\end{proof}

\end{document}